\documentclass[11pt,a4paper]{amsart}

\usepackage{amsmath,amssymb,amsthm,mathrsfs,comment,url,mathtools,enumerate,booktabs}
\usepackage[backref=page]{hyperref}
\usepackage[alphabetic]{amsrefs}
\usepackage{tikz}
\usetikzlibrary{cd}
\usepackage{enumitem}

\newcommand{\gmail}{\url{matsumoto.yuya.m@gmail.com}}
\newcommand{\tusrsmail}{\url{matsumoto_yuya@rs.tus.ac.jp}}
\newcommand{\tusaddressfull}{Department of Mathematics, Faculty of Science and Technology, Tokyo University of Science, 2641 Yamazaki, Noda, Chiba, 278-8510, Japan}

\newcommand{\cO}{\mathcal O}
\newcommand{\cL}{\mathcal L}
\newcommand{\cB}{\mathcal B}
\newcommand{\cR}{\mathcal R}
\newcommand{\cS}{\mathcal S}
\newcommand{\cI}{\mathcal I}
\newcommand{\cX}{\mathcal X}
\newcommand{\bA}{\mathbb A}
\newcommand{\bP}{\mathbb P}
\newcommand{\bZ}{\mathbb Z}
\newcommand{\bQ}{\mathbb Q}
\newcommand{\bC}{\mathbb C}
\newcommand{\bG}{\mathbb G}
\newcommand{\bF}{\mathbb F}
\newcommand{\bN}{\mathbb N}
\newcommand{\bR}{\mathbb R}
\newcommand{\setN}{\bN}
\newcommand{\setZ}{\bZ}
\newcommand{\setQ}{\bQ}
\newcommand{\setR}{\bR}
\newcommand{\setC}{\bC}
\newcommand{\setF}{\bF}

\newcommand{\idealm}{\mathfrak{m}}

\DeclareMathOperator{\Spec}{Spec}
\DeclareMathOperator{\charac}{char}
\DeclareMathOperator{\Pic}{Pic}
\DeclareMathOperator{\ord}{ord}
\DeclareMathOperator{\Aut}{Aut}
\DeclareMathOperator{\Frac}{Frac}
\DeclareMathOperator{\Ker}{Ker}
\DeclareMathOperator{\Image}{Im}
\DeclareMathOperator{\pr}{pr}
\DeclareMathOperator{\Supp}{Supp}
\DeclareMathOperator{\rank}{rank}
\DeclareMathOperator{\disc}{disc}
\DeclareMathOperator{\height}{ht}
\DeclareMathOperator{\Hom}{Hom}

\DeclareMathOperator{\Km}{Km}
\DeclareMathOperator{\Ortho}{O}
\DeclareMathOperator{\Sing}{Sing}
\DeclareMathOperator{\Fix}{Fix}
\DeclareMathOperator{\Ext}{Ext}
\newcommand{\dlog}{\mathop{d \log}}
\DeclareMathOperator{\prank}{\mathit{p}-rank}
\DeclareMathOperator{\Zero}{Zero}
\DeclareMathOperator{\sSpec}{\mathcal{S}\mathit{pec}}

\newcommand{\et}{\mathrm{\acute et}}
\newcommand{\crys}{\mathrm{crys}}
\newcommand{\id}{\mathrm{id}}
\newcommand{\sm}{\mathrm{sm}}
\newcommand{\red}{\mathrm{red}}
\newcommand{\Ga}{\mathrm{\bG_a}}
\newcommand{\Het}{H_\et}
\newcommand{\Hcrys}{H_\crys}
\newcommand{\Hfl}{H_{\mathrm{fl}}}
\newcommand{\isomto}{\stackrel{\sim}{\to}}
\newcommand{\injto}{\hookrightarrow}

\newcommand{\abs}[1]{\lvert #1 \rvert}
\newcommand{\spanned}[1]{\langle #1 \rangle}
\newcommand{\floor}[1]{\lfloor #1 \rfloor}
\newcommand{\ceil}[1]{\lceil #1 \rceil}
\newcommand{\card}[1]{\lvert #1 \rvert}
\newcommand{\restrictedto}[1]{\rvert_{#1}}
\newcommand{\set}[1]{\{#1\}}
\newcommand{\pairing}[2]{\langle #1,#2 \rangle}
\newcommand{\map}[4][\to]{#2 \colon #3 #1 #4}

\newcommand{\namedmapandmapsto}[6][\to]{#2 \colon #3 #1 #4 \colon #5 \mapsto #6}
\newcommand{\namedto}[1]{\xrightarrow{#1}}
\newcommand{\thpower}[2]{{#1}^{(#2)}}
\newcommand{\pthpower}[1]{\thpower{#1}{p}}

\newcommand{\isolatedfix}[1]{\langle #1 \rangle}
\newcommand{\divisorialfix}[1]{(#1)}
\newcommand{\normalization}[1]{#1^{\mathrm{n}}}
\newcommand{\dual}[1]{#1^{\vee}}
\newcommand{\partialdd}[1]{\frac{\partial}{\partial #1}}
\newcommand{\functorspace}{{\mathord{-}}}

 \theoremstyle{plain}
 \newtheorem{thm}{Theorem}[section]
 \newtheorem{lem}[thm]{Lemma}
 \newtheorem{prop}[thm]{Proposition}
 \newtheorem{cor}[thm]{Corollary}
 \newtheorem{claim}[thm]{Claim}
 \theoremstyle{definition}
 \newtheorem{rem}[thm]{Remark}
 \newtheorem{defn}[thm]{Definition}
 \newtheorem{example}[thm]{Example}

\numberwithin{equation}{section}

\newcommand{\Weyl}[1]{\mathrm{Weyl}(#1)}
\newcommand{\Roots}[1]{\mathrm{Roots}(#1)}

\title{Inseparable Kummer surfaces}

\author{Yuya Matsumoto}
\date{2024/03/05}
\address{\tusaddressfull}
\email{\gmail}
\email{\tusrsmail}
\thanks{This work was supported by JSPS KAKENHI Grant Numbers JP16K17560 and JP20K14296.}

\begin{document}
\begin{abstract}
We introduce an inseparable version of Kummer surfaces. 
It is defined as a supersingular K3 surface in characteristic 2 with 16 smooth rational curves 
forming a certain configuration and satisfying a suitable divisibility condition. 
The main result is that such a surface admits an inseparable double covering by a non-normal surface $A$ 
that is similar to abelian surfaces in two aspects: 
its numerical invariants are the same as abelian surfaces, 
and its smooth locus admits an abelian group structure.
\end{abstract}
\maketitle

\section{Introduction}

\subsection{Background}

Kummer surfaces are one of the most classical examples of K3 surfaces.
The Kummer surface $\Km(A)$ attached to an abelian surface $A$ is the minimal resolution $\tilde{X} \to X$ of $X := A / G$,
where $G = \setZ/2\setZ$ acts on $A$ by inversion.
Then $\tilde{X}$ is a K3 surface, precisely unless $A$ is a supersingular abelian surface in characteristic $2$.
It follows from this that there are supersingular Kummer K3 surfaces in every characteristic $\geq 3$ but not in characteristic $2$.

We have been searching an extension or an alternative notion that contains suitable 
supersingular K3 surfaces in characteristic $2$
and, hopefully, is compatible with specialization of abelian surfaces in family.

\subsection{Main idea}

Suppose a K3 surface $\tilde{X}$ is a Kummer surface as above.
We focus on the $16$ exceptional curves of $\tilde{X} \to X$.
There are $3$ possibilities for the configuration of these curves,
depending on whether $\charac(k) = 2$ or not and, when it is $2$, also on the $p$-rank of the abelian surface $A$.
Moreover, the divisor classes of these curves in $\Pic(\tilde{X})$ satisfy a certain divisibility condition, which again depends on the characteristic and the $p$-rank.
Conversely, under a suitable assumption on the characteristic and the height of the K3 surface, 
$16$ smooth rational curves in such a configuration satisfying such a divisibility condition can recover the abelian surface
(see Section \ref{subsec:usual Kummer surfaces}).

The main idea of this paper is to define an \emph{inseparable Kummer surface} to be a supersingular K3 surface in characteristic $2$ equipped with 
$16$ smooth rational curves in such a configuration and satisfying such a divisibility condition.
We say that such $16$ curves form a \emph{Kummer structure}.
The main result is that an inseparable Kummer surface admits a canonical inseparable double covering
that has properties similar to abelian surfaces, most importantly a group structure,
and we call it the \emph{abelian-like covering}.

The main result is based on lattice-theoretic properties of \emph{Kummer lattices},
introduced and discussed in Section \ref{sec:Kummer lattices},
which are certain lattices encapsulating the conditions on the $16$ curves forming Kummer structures,
and also on the computation of the spaces of sections of the sheaves $B_n \Omega^1$ outside the $16$ curves.

\subsection{Results}

Let $\tilde{X}$ be a supersingular K3 surface in characteristic $2$ 
and $C_1, \dots, C_{16}$ be smooth rational curves forming a Kummer structure.
Let $\tilde{X} \to X$ be the contraction of the $16$ curves.

We first show (Proposition \ref{prop:covering of inseparable Kummer surface}) that $X$ is isomorphic to $A/G$,
where $G$ is $\mu_2$ or $\alpha_2$, depending on the type of the Kummer structure,
and $A$ is a surface sharing the following numerical properties with abelian surfaces:
it is Gorenstein, its dualizing sheaf $\omega_A$ is isomorphic to the structure sheaf $\cO_A$,
and $h^i(\cO_A) = 1, 2, 1$ for $i = 0, 1, 2$.
The arguments needed to prove this are essentially given in our previous paper \cite{Matsumoto:k3alphap}.
Unlike abelian surfaces, however, our $A$ always has $1$-dimensional singular locus and hence is non-normal.

One particular property of abelian surfaces is, of course, the group structure.
We show (Theorem \ref{thm:group structure of abelian-like covering}) that
the smooth locus $A^{\sm}$ of $A$ admits a structure of a commutative group variety isomorphic to $\Ga^{\oplus 2}$,
and moreover the fixed locus of the action of $G$ on $A^{\sm}$ is a finite subgroup scheme of $A^{\sm}$.

Moreover, we give explicit projective equations of the surfaces $X$ as above (Theorem \ref{thm:projective equation of supersingular Kummer}).
Actually, we derive the group structure from these equations.

We also determine $\dim H^0(X^{\sm}, B_n \Omega^1)$ (Corollary \ref{cor:dim of B_n of RDP K3}) for general RDP K3 surfaces $X$, 
from which we obtain upper bounds of $\dim H^0(X^{\sm}, Z_{\infty} \Omega^1)$ (Theorem \ref{thm:characterization of Kummer types using Z_infty}) 
used in the proof of Theorem \ref{thm:projective equation of supersingular Kummer}.
The computation of $\dim H^0(X^{\sm}, B_n \Omega^1)$
is based on calculations on Witt vector-valued local cohomology groups of RDPs, 
which are essentially done in our previous paper \cite{Liedtke--Martin--Matsumoto:RDPtors}.

We can also determine which supersingular K3 surface admits a Kummer structure of which type.
This depends only on the Artin invariant of the supersingular K3 surface (Theorem \ref{thm:characterization of supersingular Kummer}).

\subsection{Relations with earlier works}

Coverings by possibly non-normal surfaces appear naturally in the study of surfaces in positive characteristic. 
For example, Bombieri--Mumford \cite{Bombieri--Mumford:III}*{Proposition 9} showed that 
there exist simply-connected Enriques surfaces in characteristic $2$,
and the canonical $G$-covering (where $G$ is $\mu_2$ or $\alpha_2$) of such a surface is possibly non-normal.

Schr\"oer \cite{Schroer:Kummer 2} studied K3 surfaces in characteristic $2$ that can be expressed as the minimal resolution of 
the quotient $(C \times C) / G$ of the self-product of a cuspidal rational curve $C$ by an action of $G = \alpha_2$ of a certain form.
Kondo--Schr\"oer \cite{Kondo--Schroer:Kummer 2} considered a more general case where $G \in \set{\mu_2, \alpha_2}$.
We will see that our framework generalizes their construction: their covering $C \times C$ is, up to a slight birational transform,
an example of an abelian-like covering $A$ in our terminology. See Section \ref{subsec:comparison with Kondo--Schroer} for a more detailed discussion.

Our previous result in \cite{Matsumoto:kummerred} shows that
inseparable Kummer surfaces discussed in this paper naturally arise as the reduction of usual Kummer surfaces $\Km(A)$
when the reduction of $A$ is supersingular in characteristic $2$.
See Section \ref{subsec:supersingular reduction} for details.
Thus, roughly speaking,
our inseparable Kummer surfaces may be viewed as a limit of usual Kummer surfaces.

\subsection{Structure of this paper}

In Section \ref{sec:preliminary} we recall preliminary notions on
derivations, Cartier operator, lattices, rational double points, and divisors and Picard lattices of K3 surfaces.

In Section \ref{sec:Kummer lattices}
we introduce Kummer lattices.
After preparing several results on overlattices of $A_1^{\oplus m}$,
we give definitions and characterizations of Kummer lattices.
Then we prove a remarkable property on their embeddings to negative-definite lattices (Proposition \ref{prop:embedding of Kummer lattices:ter}),
which gives a strong constraint on the configuration of curves on K3 surfaces realizing the lattice (Corollary \ref{cor:embedding of Kummer lattices:bis}).
We also discuss embeddings to Picard lattices of supersingular K3 surfaces in characteristic $2$.

In Section \ref{sec:Kummer structures}
we introduce Kummer structures on K3 surfaces
as $16$ smooth rational curves that induce an embedding of the Kummer lattice into the Picard lattice. 
We recall their relations with usual Kummer surfaces (Section \ref{subsec:usual Kummer surfaces}) .
We then define inseparable Kummer surfaces to be supersingular K3 surfaces in characteristic $2$ equipped with Kummer structures.
We characterize which surface admit such structures (Theorem \ref{thm:characterization of supersingular Kummer})
and show that Kummer structures induce inseparable coverings (Proposition \ref{prop:covering of inseparable Kummer surface}).

In Section \ref{sec:dimension of Z_infty}, 
we give upper bounds of the dimensions of the spaces $H^0(X^{\sm}, Z_{\infty} \Omega^1)$ 
for RDP K3 surfaces $X$ (Theorem \ref{thm:characterization of Kummer types using Z_infty}).
This result is used in the proof of Theorem \ref{thm:projective equation of supersingular Kummer}.
The upper bound is based on the values of $\dim H^0(X^{\sm}, B_n \Omega^1)$ (Corollary \ref{cor:dim of B_n of RDP K3}) and some combinatorial arguments.

In Section \ref{sec:projective equations}, 
we show that any inseparable Kummer surface
belongs to one of two families given by explicit projective equations (Theorem \ref{thm:projective equation of supersingular Kummer}).

In Section \ref{sec:structure of the covering}
we investigate the structure of the abelian-like covering $A$.
Using the projective equations given in the previous section, 
we derive the group structure on the smooth part of $A$ (Theorem \ref{thm:group structure of abelian-like covering}).
We also describe the normalization $\normalization{A}$ of $A$ (Section \ref{subsec:normalization}),
and discuss which subalgebras of $\cO_{\normalization{A}}$ appear as the abelian-like coverings of inseparable Kummer surfaces
(Section \ref{subsec:Kummer quotients}).

In Section \ref{sec:final remarks}, 
we discuss relations with Kondo--Schr\"oer's work (Section \ref{subsec:comparison with Kondo--Schroer})
and relations with good (supersingular) reduction of usual Kummer surfaces (Section \ref{subsec:supersingular reduction}).

\section{Preliminaries} \label{sec:preliminary}

In this paper we work over an algebraically closed field $k$.

\subsection{Derivations and quotients} \label{subsec:derivation}

A (regular) \emph{derivation} on a $k$-scheme $Y$ is a $k$-linear morphism $\map{D}{\cO_Y}{\cO_Y}$
satisfying the Leibniz rule $D(f g) = f D(g) + D(f) g$.
When $Y$ is integral, a \emph{rational derivation} is a $k$-linear morphism $\cO_Y \to k(Y)$ locally of the form $f D$ with $f$ a rational function and $D$ a regular derivation.

The \emph{fixed locus} of a derivation $D$, denoted by $\Fix(D)$,
is the closed subscheme (or closed subset) corresponding to the ideal generated by the vector subspace $D(\cO_X) \subset \cO_X$.

Suppose $D$ is a derivation on a surface $Y$ 
and $Y' \to Y$ is the blow-up at an RDP $x$.
Then $D$ extends to a (regular) derivation on $Y'$ if and only if $x \in \Fix(D)$,
and otherwise $D$ extends only to a rational derivation.

Suppose $Y$ is integral.
A derivation $D$ is \emph{$p$-closed} if there exists $h \in k(Y)$ satisfying $D^p = h D$.
If $D$ is $p$-closed and nonzero, then $Y \to Y^D$ is generically of degree $p$.
Here $Y^D$ is the quotient of $Y$ by $D$ and is described as follows:
the underlying topological space is the same as $Y$, and the structure sheaf is $(\cO_Y)^D = \set{a \in \cO_Y \mid D(a) = 0}$.

If $D$ is a derivation on a variety $Y$ regular in codimension $1$, 
then $\divisorialfix{D}$ is defined to be the Weil divisor corresponding to $\Fix(D)$ at each point of codimension $1$.
It is used in the Rudakov--Shafarevich formula to compare the canonical divisors of $Y$ and $Y^D$.
However, the formula is not useful in this paper 
since we do not assume $Y$ to be regular in codimension $1$.

Actions of $\mu_p$ and $\alpha_p$ correspond to derivations of certain type.

\begin{prop}
Let $Y$ be a scheme over a field $k$ of characteristic $p > 0$.
\begin{enumerate}
\item 
There is a canonical correspondence between the following notions:
\begin{itemize}
\item Derivations $D$ on $Y$ of multiplicative type (i.e.\ $D^p = D$).
\item Actions of $\mu_p$ on $Y$.
\item $\setZ/p\setZ$-grading structures of $\cO_Y$,
i.e., decompositions $\cO_Y = \bigoplus_{i \in \setZ/p\setZ} (\cO_Y)_i$ satisfying $1 \in (\cO_Y)_0$ and $(\cO_Y)_i (\cO_Y)_j \subset (\cO_Y)_{i+j}$.
\end{itemize}
\item 
There is a canonical correspondence between the following notions:
\begin{itemize}
\item Derivations $D$ on $Y$ of additive type (i.e.\ $D^p = 0$).
\item Actions of $\alpha_p$ on $Y$.
\end{itemize}
\end{enumerate}
Moreover, the quotient by the derivation and the quotient by the group scheme action are equal and,
in the case of $\mu_p$, also equal to $(\cO_Y)_0$.
\end{prop}
\begin{proof}
The bijection between derivations and group scheme actions follows from
\cite{SGA3-1}*{Th\'eor\`eme VII.7.2(ii)}.

The correspondence between the derivation of multiplicative type and the $\setZ/p\setZ$-grading is that $(\cO_Y)_i$ is the eigenspace of $D$ with eigenvalue $i \in \setF_p$.
\end{proof}

\begin{thm}[Katsura--Takeda formula \cite{Katsura--Takeda:quotients}*{Proposition 2.1}] \label{thm:Katsura--Takeda}
	Let $D$ be a rational derivation on a smooth proper surface $Y$. 
	Then
	\[
	\deg c_2(Y) = \deg \isolatedfix{D} - K_Y \cdot \divisorialfix{D} - \divisorialfix{D}^2.
	\]
\end{thm}

If $\Fix(D) = \emptyset$, then we say that $Y \to X$ is a \emph{$G$-torsor}.

\begin{prop}[\cite{Matsumoto:k3alphap}*{Lemma 2.11}] \label{prop:1-form corresponding to G-torsor}
Suppose $X$ is a smooth variety and 
$Y \to X$ is a $G$-torsor ($G \in \set{\mu_p, \alpha_p}$) with corresponding derivation $D$.
Then there is a canonical global section $\eta$ of $\Ker(\Omega^1_X \to \pi_* \Omega^1_Y)$
characterized by the equality $D(t)^p \eta = d(t^p)$ for any $t \in k(Y)$. 
In particular,
\begin{itemize}
\item If $G = \mu_p$ and $t \in \cO_Y^*$ satisfies $D(t) = t$, then $\eta = \dlog(t^p)$.
\item If $G = \alpha_p$ and $t \in \cO_Y$ satisfies $D(t) = 1$, then $\eta = d(t^p)$.
\end{itemize}

Conversely, a $1$-form $\eta$ on a smooth variety $X$ that is locally of the form 
$\eta = \dlog(f)$ with $f \in \cO_X^*$ (resp.\ $\eta = df$ with $f \in \cO_X$)
is induced by a unique $\mu_p$-torsor (resp.\ $\alpha_p$-torsor).
\end{prop}

We have the following explicit relation between non-normal torsors and their normalizations if $p = 2$.

\begin{prop} \label{prop:normalization}
Let $\map{\pi}{Y}{X}$ be a $G$-torsor over a smooth variety $X$, where $G$ is $\mu_2$ or $\alpha_2$.
Let $\Delta$ be the pullback of the divisorial part of $\Zero(\eta)$ to the normalization $\normalization{Y}$ of $Y$.
Then we have 
$\cO_Y = \cO_X + \cO_{\normalization{Y}}(-\Delta)$.
Here, the sum is taken as subsheaves of $k$-vector spaces of $\cO_{\normalization{Y}}$.
\end{prop}
\begin{proof}
Locally, we can write $\cO_Y = \cO_X[\sqrt{f}]$ with $f \in \cO_X$ and $D(\sqrt{f}) \in \cO_Y^*$.
We have $\cO_{\normalization{Y}} = \cO_X \oplus \cO_X \cdot {(a + \sqrt{f})}/{g}$ for some $a, g \in \cO_X$.

For any choice of local coordinates $x_1, \dots, x_m$ (where $m = \dim X$),
we write $f = \sum_{i_1, \dots, i_m \in \set{0, 1}} f_{(i_1, \dots, i_m)} x_1^{i_1} \dots x_m^{i_m}$
with $f_{(i_1, \dots, i_m)} \in \thpower{\cO_X}{2}$.
Then $g = \gcd\set{f_{(i_1, \dots, i_m)} \mid (i_1, \dots, i_m) \neq (0, \dots, 0)}$,
whereas the divisorial part of $\Zero(\eta)$ is generated by $g' := \gcd\set{\partial f / \partial {x_1}, \dots, \partial f / \partial {x_m}}$ 
because $\eta = df / D(\sqrt{f})^2$.
It remains to show $g = g'$.

Clearly, $g$ divides $g'$. 
We show that $\ord_E(g) = \ord_E(g')$ for any prime divisor $E \subset X$.
We may assume, by restricting to a neighborhood of a smooth point of $E$, that $E$ is defined by $(x_1 = 0)$.
Take an index $(i_1, \dots, i_m)$ minimizing $\ord_{x_1} f_{(i_1, \dots, i_m)}$.
If $i_1 = 1$, then $\ord_{x_1} (\partial f / \partial {x_1})$ attains the same value.
If $i_1 = 0$, then take $j$ such that $i_j = 1$,
and then $\ord_{x_1} (\partial f / \partial {x_j})$ attains the same value.
\end{proof}

\subsection{Cartier operator on differential forms} \label{subsec:Cartier operator}

In this paper we need only the Cartier operator for $1$-forms.
Let $X$ be a smooth variety in characteristic $p > 0$.
Let $\Omega_{X,\text{closed}}^1 = \Ker(\map{d}{\Omega_X^1}{\Omega_X^2})$.
The Cartier operator $\map{C}{\Omega_{X,\text{closed}}^1}{\Omega_X^1}$
satisfies the following properties:
\begin{itemize}
\item $C$ is additive and satisfies $C(f^p \omega) = f C(\omega)$ for $f \in \cO_X^1$.
\item $C(\eta) = 0$ if and only if $\eta$ is locally of the form $df$.
\item $C(\eta) = \eta$ if and only if $\eta$ is locally of the form $\dlog f$.
\end{itemize}

The sequences $B_i \Omega_X^1, Z_i \Omega_X^1 \subset \Omega_X^1$ ($i \geq 0$) of subsheaves are defined by 
$B_0 \Omega_X^1 = 0$,
$Z_0 \Omega_X^1 = \Omega_X^1$, 
$B_{i+1} \Omega_X^1 = C^{-1}(B_i \Omega_X^1)$, 
and $Z_{i+1} \Omega_X^1 = C^{-1}(Z_i \Omega_X^1)$ for $i \geq 0$.
Finally, we let $B_{\infty} \Omega_X^1 = \bigcup_{i \geq 0} B_i \Omega_X^1$ 
and $Z_{\infty} \Omega_X^1 = \bigcap_{i \geq 0} Z_i \Omega_X^1$.
We have
\[
0 = B_0 \Omega_X^1 \subset B_1 \Omega_X^1 \subset \dots \subset B_{\infty} \Omega_X^1 
\subset Z_{\infty} \Omega_X^1 \subset \dots \subset Z_1 \Omega_X^1 \subset Z_0 \Omega_X^1 = \Omega_X^1.
\]

Combining these with results in Section \ref{subsec:derivation}, we obtain the following.
\begin{prop} \label{prop:torsors and 1-forms}
Let $X$ be a smooth variety with $H^0(X, \cO_X) = k$.
\begin{enumerate}
\item There is a canonical correspondence between the following objects.
\begin{itemize}
\item Isomorphism classes of $\mu_p$-torsors over $X$.
\item Classes of $\Pic(X)$ annihilated by $p$.
\item $1$-forms locally of the form $\dlog(f)$ ($f \in \cO_X^*$).
\item Elements of $H^0(X, Z_{\infty} \Omega^1_X)^{C = 1}$.
\end{itemize}
\item There is a canonical correspondence between the following objects.
\begin{itemize}
\item Isomorphism classes of $\alpha_p$-torsors over $X$.
\item Classes of $H^1(X, \cO_X)$ annihilated by $F$.
\item $1$-forms locally of the form $df$ ($f \in \cO_X$).
\item Elements of $H^0(X, Z_{\infty} \Omega^1_X)^{C = 0} = H^0(X, B_1 \Omega^1_X)$.
\end{itemize}
\end{enumerate}
\end{prop}
\begin{proof}
Isomorphism classes of torsors correspond to $\Hfl^1$ of the group scheme, 
which can be computed by using exact sequences
$1 \to \mu_p \to \cO_X^* \namedto{p} \cO_X^* \to 1$ and  
$0 \to \alpha_p \to \cO_X \namedto{F} \cO_X \to 0$
(of fppf sheaves).
Other assertions follow from Proposition \ref{prop:1-form corresponding to G-torsor} and the properties of $C$ described above.
\end{proof}

We also recall the following.
\begin{prop} 
Let $\map{C}{V}{V}$ be a semilinear endomorphism of a finite-dimensional $k$-vector space $V$
(in the sense that $C$ is additive and $C(c v) = c^{1/p} C(v)$ for $c \in k$).
Then $V$ admits a canonical decomposition $V = V_{\textrm{ss}} \oplus V_{\textrm{nilp}}$ into $C$-stable subspaces
such that $C$ acts on $V_{\textrm{nilp}}$ nilpotently
and $V_{\textrm{ss}}$ admits a basis $e_i$ such that $C(e_i) = e_i$.
Moreover, $V^{C = 1} := \set{v \in V \mid C(v) = v}$ is an $\setF_p$-vector space of dimension equal to $\dim_k V_{\textrm{ss}}$.
\end{prop}

We can apply this to the Cartier operator on $V = H^0(X, Z_{\infty} \Omega_X^1)$, assuming this space is finite-dimensional, 
and then $V_{\textrm{nilp}} = H^0(X, B_{\infty} \Omega_X^1)$
and $V_{\textrm{ss}} \cong H^0(X, Z_{\infty} \Omega_X^1) / H^0(X, B_{\infty} \Omega_X^1)$.

We will also use (the case of $p = 2$ of) the following explicit description.

\begin{lem} \label{lem:explicit formula for Cartier operator:general}
Suppose $\charac k = p > 0$.
Let $X = \Spec k[x, y, w] / (w^p - H(x, y))$, where $H \in k[x, y]$ satisfies $H \notin k[x^p, y^p]$,
and let $\eta_0 = dx/H_y = - dy/H_x \in H^0(X^{\sm}, \Omega^1)$.
Then $k[x, y] \cdot \eta_0 \subset H^0(X^{\sm},  Z_1 \Omega^1)$
and, for $g \in k[x, y]$, we have
\[
C(g \eta_0) = \biggl( \sum_{a = 0}^{p-1} w^{p-1-a} \sqrt[p]{(g H^a)_{(p-1,p-1)}} \biggr) \eta_0,
\]
where, for $J \in k[x, y]$, we define $J_{(i ,j)} \in k[x^p, y^p]$ ($0 \leq i, j \leq p-1$) by 
$J = \sum_{i, j} J_{(i, j)} x^i y^j$.
In other words, $J_{(p-1, p-1)} = (\partial/\partial x)^{p-1} (\partial/\partial y)^{p-1} J$.
\end{lem}

In particular, if $p = 2$ we have
$C(g \eta_0) = \bigl( w \sqrt{g_{x y}} + \sqrt{(g H)_{x y}} \bigr) \eta_0$.

\begin{proof}
Let $\map{C'}{k(x, y) \eta_0}{\Omega^1_{k(X)/k}}$ be the map defined by the stated formula.
We have $C'(t^p \eta) = t C'(\eta)$ for $t \in k(X)$ (this is clear if $t \in k(x, y)$ and easily checked if $t = w$).
Hence, to show $C' = C \restrictedto{k(x, y) \eta_0}$, it suffices to check the equality for 
a set of generators of the $\pthpower{k(X)}$-vector space $k(x, y) \eta_0$.
We show that
\begin{enumerate}
\item \label{item:C' df} $C'(dF) = 0$ for $F \in k(x, y)$, and 
\item \label{item:C' dlog f} $C'(F^{p-1} dF) = dF$ for $F \in k(x, y)$.
\end{enumerate}

In fact we show that, for integers $0 \leq a, b \leq p-1$,
\[
\partial_x^{p-1} \partial_y^{p-1} ((F_x H_y - F_y H_x) F^b H^a) = 
\begin{cases}
(F_x H_y - F_y H_x)^p & ((a, b) = (p-1, p-1)), \\
0 & (\text{otherwise}),
\end{cases}
\]
where we write $\partial_x := \partial/\partial x$ and $\partial_y := \partial/\partial y$.
Then we obtain (\ref{item:C' df}) and (\ref{item:C' dlog f}) by letting $b = 0$ and $b = p-1$ respectively
(note that $dF = (F_x H_y - F_y H_x) \eta_0$).

Consider
\[
\partial_x^{p-1} \partial_y^{p-1} (F_x H_y F^b H^a) = \sum c \cdot 
\partial_x^{i_1} \partial_y^{j_1} (F_x F^b) \cdot 
\partial_x^{i_2} \partial_y^{j_2} (H_y H^a),
\]
where the sum is taken over the tuples $(i_1, i_2, j_1, j_2)$ of non-negative integers with $i_1 + i_2 = j_1 + j_2 = p-1$,
and $c \in \setF_p$ are coefficients depending on the tuples.
By Lemma \ref{lem:p-1 derivative},
the summand is zero in any of the following cases.
\begin{itemize}
\item $i_1 = p-1$ and $b < p-1$.
\item $i_1 = p-1$ and $j_1 > 0$.
\item $j_2 = p-1$ and $a < p-1$.
\item $j_2 = p-1$ and $i_2 > 0$.
\end{itemize}
If $i_1 = j_2 = a = b = p-1$,
then the summand is $F_x^p H_y^p$.
The remaining terms, which are precisely the ones with $i_2 > 0$ and $j_1 > 0$, can be written as 
$c \cdot \partial_x^{i_1} \partial_y^{j_1-1} (F_{xy} F^b + b F_x F_y F^{b-1}) \cdot 
\partial_x^{i_2-1} \partial_y^{j_2} (H_{xy} H^a + a H_x H_y H^{a-1})$.
We observe that these terms are symmetric on $x$ and $y$ 
and hence cancel with the corresponding terms coming from $\partial_x^{p-1} \partial_y^{p-1} (F_y H_x F^b H^a)$.
\end{proof}

\begin{lem} \label{lem:p-1 derivative}
Let $F$ be a polynomial in one variable $t$ over a ring in characteristic $p > 0$.
Then $(d/dt)^{p-1} (F_t F^a) = 0$ for any integer $0 \leq a \leq p-2$,
and $(d/dt)^{p-1} (F_t F^{p-1}) = -F_t^p$.
\end{lem}
\begin{proof}
If $0 \leq a \leq p-2$, then this is clear since $F_t F^a = (d/dt) (a+1)^{-1} F^{a+1}$.
Consider the case $a = p-1$.
Let $E$ be a polynomial in one variable $t$ over a $p$-torsion-free ring.
We have 
\[
p (d/dt)^{p-1} E_t E^{p-1} = (d/dt)^p E^p = \sum c \cdot ((d/dt)^{b_1} E) \dots ((d/dt)^{b_p} E),
\]
where the sum is over the non-increasing $p$-tuples $b_1 \geq \dots \geq b_p \geq 0$ of non-negative integers satisfying $b_1 + \dots + b_p = p$,
and $c \in \setZ$ are coefficients depending on the tuples.
We have $c = p$ if $(b_1, \dots, b_p) = (p, 0, \dots, 0)$,
$c = p! \equiv -p \pmod{p^2}$ if $(b_1, \dots, b_p) = (1, \dots, 1)$, 
and $c \equiv 0 \pmod{p^2}$ if otherwise.
Dividing by $p$ and then going to characteristic $p$, we obtain the stated formula.
\end{proof}

\subsection{Lattices}

A \emph{lattice} is a free $\setZ$-module $L$ of finite rank equipped with a symmetric bilinear pairing 
$\map{\pairing{\functorspace}{\functorspace}}{L \times L}{\setZ}$.
We also use the multiplicative notation for the pairing.

A lattice $L$ is \emph{even} if $v^2 \in 2 \setZ$ for every $v \in L$
and \emph{odd} if otherwise.

The \emph{dual} of a lattice $L$ is $\dual{L} := \Hom(L, \setZ)$.
We say that $L$ is \emph{non-degenerate} (resp.\ unimodular) if
the natural map $L \to \dual{L}$ is injective (resp.\ an isomorphism).
If $L$ is non-degenerate, 
then the cokernel is finite,
$\dual{L}$ can be viewed as a submodule of $L \otimes_{\setZ} \setQ$,
and the pairing on $L$ extends to a $\setQ$-valued symmetric bilinear pairing on $\dual{L}$.

If $L$ is even and non-degenerate,
then the quadratic map $\namedmapandmapsto{q}{\dual{L}}{\setQ}{x}{x^2}$ induces a quadratic map $\map{\bar{q}}{\dual{L}/L}{\setQ/2\setZ}$.
Here, a map $q$ from an abelian group to $\setQ$ (resp.\ $\setQ/2\setZ$) is a quadratic map
if $q(n x) = n^2 q(x)$ for $n \in \setZ$ holds and 
$(q(x + y) - q(x) - q(y))/2$ is a $\setQ$-valued (resp.\ $\setQ/\setZ$-valued) bilinear map.

The \emph{Gram matrix} of a lattice with respect to a basis $e_1, \dots, e_n$ 
is the matrix $(e_i \cdot e_j)_{1 \leq i, j \leq n}$.
The lattice is non-degenerate if and only if the Gram matrix has nonzero determinant,
in which case the determinant is called the \emph{discriminant} of the lattice and denoted by $\disc(L)$.
The discriminant does not depend on the choice of a basis, and its absolute value is equal to the order of $\dual{L}/L$.

A non-degenerate lattice $L$ has \emph{signature} $(r, s)$
if $L \otimes_{\setZ} \setR$ admits an orthogonal basis $(e_i)_{i = 1}^{r + s}$
with $e_i^2 = 1$ for $1 \leq i \leq r$ and $e_i^2 = -1$ for $r + 1 \leq i \leq r + s$.
We say that $L$ is \emph{positive-definite} (resp.\ \emph{negative-definite}) if $s = 0$ (resp.\ $r = 0$).

A \emph{sublattice} of a lattice is simply a submodule, equipped with the same pairing.
An \emph{overlattice} is the opposite relation of a sublattice.

The \emph{saturation} (or \emph{primitive closure}) of a sublattice $M$ in $L$ is the unique overlattice $\bar{M}$ of $M$ contained in $L$
such that $\bar{M}/M$ is finite and $L/\bar{M}$ is torsion-free.
It can be written as $\bar{M} = L \cap (M \otimes_{\setZ} \setQ)$.
We say that $M$ is \emph{saturated} (or \emph{primitive}) in $L$ if its saturation is itself.

A \emph{root} of a lattice $L$ is an element $v \in L$ with $v^2 = -2$.
We denote by $\Roots{L}$ the set of roots.
The \emph{root sublattice} of a lattice $L$ is the sublattice generated by its roots.

The root lattices of type $A_n, D_n, E_n$ are denoted by the same symbols, and are negative-definite by convention.
We call them \emph{ADE lattices}.

The \emph{reflection} with respect to a root $v \in L$ 
is the isometry $\namedmapandmapsto{r_v}{L}{L}{x}{x + \pairing{x}{v} v}$.
It satisfies $r_v^2 = \id$.
The \emph{Weyl group} of $L$, denoted by $\Weyl{L}$, 
is the subgroup of $\Ortho(L)$ generated by $\set{r_v \mid \text{$v$ is a root of $L$}}$.

For a lattice $L = (L, \pairing{\functorspace}{\functorspace})$ and a rational number $q$, 
we denote $(L, q \pairing{\functorspace}{\functorspace})$ by $L(q)$.

Suppose $p$ is a prime. A lattice $L$ is \emph{$p$-elementary}
if $\dual{L}/L$ is annihilated by $p$ (not only by a power of $p$).
If $L$ is $2$-elementary, then we say that $L$ is of \emph{type 2} if $x^2 \in \setZ$ for all $x \in \dual{L}$,
and \emph{type 1} if otherwise.

\subsection{Non-taut rational double points and relation to the height of K3 surfaces}

Let $X$ be a K3 surface (resp.\ an abelian surface) in positive characteristic.
The \emph{height} of $X$ is an invariant taking values in $\set{1, 2, \dots, 10} \cup \set{\infty}$ (resp.\ $\set{1, 2} \cup \set{\infty}$).
One of equivalent characterizations of the height is that 
the slopes of $\Hcrys^2(X/W(k))$ are $1$ if $\height(X) = \infty$
and $1 - 1/h, 1, 1 + 1/h$ if $\height(X) = h < \infty$.
We use the following facts.

\begin{prop}[\cite{Illusie:deRham--Witt}*{Section II.7.1}] \label{prop:height of abelian surfaces}
Let $X$ be an abelian surface in characteristic $p > 0$.
Then $\height(X) = 1, 2, \infty$ if and only if $\prank(X) = 2, 1, 0$ respectively.
\end{prop}

\begin{prop} \label{prop:height and Picard number}
Let $X$ be a K3 surface. Let $\rho$ be its Picard number (the rank of the free $\setZ$-module $\Pic(X)$) and $\height(X)$ its height.
\begin{enumerate}
\item 
If $\height(X) < \infty$, then $\rho \leq 22 - 2 \height(X)$.
\item $\height(X) = \infty$ if and only if $\rho = 22$.
In this case, $X$ is called \emph{supersingular}.
\end{enumerate}
\end{prop}
\begin{proof}
In the finite height case, this follows from the characterization of the height using the slopes of the crystalline cohomology.
See \cite{Illusie:deRham--Witt}*{Section II.7.2}.

In the infinite height case, this is the Tate conjecture for supersingular K3 surfaces, proved by 
Madapusi Pera \cite{MadapusiPera:TateK3}*{Theorem 1} for characteristic $\geq 3$,
and by Kim--Madapusi Pera \cite{Kim--MadapusiPera:2-adic}*{Theorem A.1} and 
Ito--Ito--Koshikawa \cite{Ito--Ito--Koshikawa:CM liftings}*{Remarks 6.9 and 6.10}
for characteristic $2$.
\end{proof}

\emph{Rational double points} (\emph{RDPs} for short) are the $2$-dimensional canonical singularities.

The exceptional curves of the minimal resolution of an RDP
are smooth rational curves, whose dual graph 
is a Dynkin diagram of type $A_n$, $D_n$, or $E_n$,
and the RDP is said to be of type $A_n$, $D_n$, or $E_n$ accordingly.

We say that a surface is an \emph{RDP surface} if its singularities (if any) are all RDPs.
An \emph{RDP K3 surface} is a proper RDP surface whose minimal resolution is a K3 surface.
We define the height of an RDP K3 surface to be the height of its minimal resolution.

In most characteristics RDPs are taut, that is, if two RDPs are of the same type then the singularities are formally isomorphic.
However there exist non-taut RDPs in characteristic $2$, $3$, and $5$.
Artin \cite{Artin:RDP} classified non-taut RDPs and introduced the notation $D_n^r$ and $E_n^r$,
where the coindex $r$ takes non-negative integer values up to a bound depending on the characteristic and the dual graph.
However, in this paper we use convention that the range of $r$ in $D_{2l+1}^r$ (in characteristic $2$)
is $\set{1/2, 3/2, \dots, m-1/2}$ as in \cite{Matsumoto:rdpderi}*{Convention 1.2}, 
which is also used in \cite{Liedtke--Martin--Matsumoto:RDPtors} cited below,
instead of $\set{0, 1, \dots, m-1}$ as in Artin's notation.

In \cite{Matsumoto:k3rdpht} we gave relations between the isomorphism class of a non-taut RDP on an RDP K3 surface and the height of the surface.
The following subset of the result is needed in this paper.

\begin{prop}[\cite{Matsumoto:k3rdpht}*{Theorem 1.2}] \label{prop:RDP and height}
Suppose $z \in X$ is a non-taut RDP on an RDP K3 surface $X$ in characteristic $p = 2$.
\begin{itemize}
\item If $z$ is of type $D_4^r$ and $\height(X) = 1$, then $r = 1$.
\item If $z$ is of type $D_8^r$ and $\height(X) = 2$, then $r = 2$.
\item If $z$ is of type $D_{2l}^r$, $D_{2l+1}^{r+1/2}$, or $E_N^r$, and $X$ is supersingular, then $r = 0$.
\end{itemize}
\end{prop}

\subsection{Picard lattices of supersingular K3 surfaces}

\begin{prop}[\cite{Rudakov--Shafarevich:finite}*{Section 1}]  \label{prop:uniqueness of p-elementary lattice}
Suppose $n > 2$ is an integer and $p \neq 2$.
An even non-degenerate $p$-elementary lattice of signature $(1, n-1)$
is uniquely determined (up to isometry) by its discriminant.

If $n > 4$ and $p = 2$, then such a lattice is uniquely determined by its discriminant and type (1 or 2).
\end{prop}

It is known that the Picard group of a K3 surface is an even non-degenerate lattice of signature $(1, \rho - 1)$
for some $1 \leq \rho \leq 22$.
A K3 surface is supersingular if and only if $\rho = 22$
(Proposition \ref{prop:height and Picard number}).
It is known (\cite{Rudakov--Shafarevich:finite}*{Section 8}) that
the Picard lattice of a supersingular K3 surface in characteristic $p > 0$ is $p$-elementary and, if $p = 2$, then it is moreover of type 2.
The \emph{Artin invariant} of a supersingular K3 surface $X$ is defined to be $\sigma \in \set{1, 2, \dots, 10}$ 
satisfying $\dual{\Pic(X)}/\Pic(X) \cong (\setZ/p\setZ)^{2 \sigma}$.

\subsection{Divisors on K3 surfaces}

Let $X$ be a K3 surface.

\begin{prop} \label{prop:-2 divisor}
If $D \in \Pic(X)$ is a divisor class with $D^2 \geq -2$,
then either $D$ or $-D$ is effective.
\end{prop}
\begin{proof}
By the Riemann--Roch theorem we have 
$h^0(\cO_X(D)) + h^2(\cO_X(D)) \geq h^0(\cO_X(D)) - h^1(\cO_X(D)) + h^2(\cO_X(D)) = 2 + (1/2)D^2 \geq 1$,
hence $h^0(\cO_X(D)) > 0$ or $h^2(\cO_X(D)) > 0$. We have $h^0(\cO_X(-D)) = h^2(\cO_X(D))$.
\end{proof}

\begin{prop}[\cite{Ogus:crystalline torelli}*{1.10}] \label{prop:nef up to Weyl}
Let $N$ be a non-degenerate sublattice of signature $(1, n-1)$, $n \geq 2$,
of the Picard lattice of a K3 surface.
If $D \in N$ is a divisor class with $D^2 \geq 2$,
then there is a unique class in the $\pm \Weyl{N}$-orbit of $D$ 
that is nef and effective.
\end{prop}

\begin{prop} \label{prop:-2 divisors orthogonal to a big nef}
If $D$ is a nef effective divisor with $D^2 > 0$
and $Z_1, \dots, Z_n$ are effective divisors with $Z_i^2 = -2$ and $Z_i \cdot D = 0$,
then $Z := \bigcup_{i = 1}^n \Supp Z_i$ is an ADE configuration.
\end{prop}
Here, by an \emph{ADE configuration} we mean
a disjoint union of finitely many connected divisors whose irreducible components are smooth rational curves and whose dual graphs are
of type $A_n$, $D_n$, or $E_n$.
\begin{proof}
Take the Zariski--Fujita decomposition $Z_i = P_i + N_i$
	with $P_i$ a nef $\setQ$-divisor and 
	$N_i$ an effective $\setQ$-divisor supported on an ADE configuration (see \cite{Badescu:surfaces}*{Theorem 14.14}).
	Since $P_i$ is nef and orthogonal to a nef big divisor $D$,
	it follows that $P_i = 0$, and hence $N_i$ is a $\setZ$-divisor.
	Then $N_i$ is connected since $N_i^2 = -2$.
	Since $D$ is nef, it is orthogonal to each irreducible component of $N_i$.
	Then all irreducible components of $Z = \bigcup_i \Supp(N_i)$ 
	are orthogonal to $D$,
	hence $Z$ is an ADE configuration.
\end{proof}

\begin{prop} \label{prop:-2 vectors}
Suppose smooth rational curves $C_1, \dots, C_k \subset X$ 
form an ADE configuration.
Let $N' \subset \Pic(X)$ be the sublattice generated by the classes $[C_i]$,
and $N$ its saturation.
Then any class $v \in N$ with $v^2 = -2$ 
is of the form $\sum_{i = 1}^k a_i [C_i]$ with either $a_i \in \setZ_{\leq 0}$ or $a_i \in \setZ_{\geq 0}$.

In particular, the type of the configuration of the curves is the same as the type of the lattice.
\end{prop}
\begin{proof}
We may assume $v$ is an effective divisor.
Let $C'_1, \dots, C'_{k'}$ be the components of $\Supp v$ different from $C_1, \dots, C_k$.
Since $N$ is negative-definite, we can take a nef effective nonzero divisor $D$ orthogonal to $N$.
Applying Proposition \ref{prop:-2 divisors orthogonal to a big nef} 
to $D$ and $C_1, \dots, C_k, v$,
we observe that $\bigcup_{i = 1}^k C_i \cup \bigcup_{j = 1}^{k'} C'_j$ is an ADE configuration.
This implies that $[C_1], \dots, [C_k], [C'_1], \dots, [C'_{k'}]$ are linearly independent.
Since $v \in N$, we obtain $k' = 0$.
\end{proof}

\section{Kummer lattices} \label{sec:Kummer lattices}

We introduce $5$ specific lattices, which we call \emph{Kummer lattices} (of type $16 A_1, 4 D_4, 2 D_8, 1 D_{16}, 2 E_8$).
We will use them in Section \ref{subsec:definition of Kummer structures} to define Kummer structures, 
and discuss their relevance to usual Kummer surfaces in Section \ref{subsec:usual Kummer surfaces}.
In this section we discuss lattice-theoretic properties of 
Kummer lattices and their embeddings to negative-definite lattices (Section \ref{subsec:embeddings of Kummer lattices to negative-definite lattices})
and to Picard lattices of supersingular K3 surfaces in characteristic $2$ (Section \ref{subsec:embeddings of Kummer lattices to supersingular K3 lattices}).

\subsection{Overlattices of $A_1^{\oplus m}$} \label{subsec:overlattices of A_1^m}

We define a function $\map{f}{\set{0, 1, \dots, 24}}{\setN}$ by
\[
f(m) = 
\begin{cases}
4 - \ceil{\log_2(16 - m)} & (m < 16), \\
5 & (m = 16), \\
m - 12 & (17 \leq m \leq 24),
\end{cases}
\]
and prove the following, although only the cases $m \leq 21$ are needed in this paper.
(This might be well-known to experts.)
\begin{prop} \label{prop:overlattices of A_1^m:bis}
Let $0 \leq m \leq 24$ be an integer.
Let $L$ be an even lattice whose root sublattice is of finite index $2^i$ and isomorphic to $A_1^{\oplus m}$.
Then $i \leq f(m)$, and for every $m$ the equality is achievable.
Moreover, such a lattice with $m = 16$ and $i = 5$ is unique up to isometry.
\end{prop}

All even overlattices of $A_1^{\oplus m}$ of finite index are contained in $\dual{(A_1^{\oplus m})} = (1/2) A_1^{\oplus m}$, 
hence correspond to subgroups $V$ of $(1/2) A_1^{\oplus m} / A_1^{\oplus m} \cong 2^S$, where $\card{S} = m$,
with the property that any element of $V$ has cardinality divisible by $4$.
Having no roots other than those of $A_1^{\oplus m}$ means that no element of $V$ has cardinality $4$.
Thus the proposition is paraphrased as follows.

\begin{prop} \label{prop:dimension of subspaces with 0 mod 4}
Let $0 \leq m \leq 24$ be an integer.
Let $g(m)$ be the maximum dimension of an $\setF_2$-vector subspace
$V \subset 2^S$, where $S$ is a set with $\card{S} = m$,
such that every $B \in V$ satisfies $\card{B} \equiv 0 \pmod{4}$ and $\card{B} \neq 4$.
Then $g(m) = f(m)$.
Moreover, if $m = 16$ and $\dim V = 5$ then $V$ is isomorphic to the one given in Example \ref{ex:5-dimensional in 16}.
\end{prop}

We obtain $g(m) \geq f(m)$ for $m = 16, 16 - 2^{4-j}, 24$ respectively by Examples \ref{ex:5-dimensional in 16}, \ref{ex:5-i-dimensional in 16-2^i}, and \ref{ex:12-dimensional in 24}.
Together with the inequalities $g(m+1) \geq g(m)$ (obvious) and 
$g(m) \geq g(m+1)-1$ (obtained by restricting to the subsets not containing one particular element),
we obtain $g(m) \geq f(m)$ for every $0 \leq m \leq 24$.

\begin{example}[$m = 16$, $\dim V = 5$] \label{ex:5-dimensional in 16}
Let $S_{16}$ be an $\setF_2$-vector space of dimension $4$ (hence $\card{S_{16}} = 16$). Then 
\[ V_{16} := \set{\emptyset} \cup \set{S_{16}} \cup \set{\text{affine hyperplanes of $S_{16}$}} \subset 2^{S_{16}} \]
gives a $5$-dimensional space satisfying the condition of Proposition \ref{prop:dimension of subspaces with 0 mod 4}.
\end{example}

\begin{example}[$m = 16 - 2^{4-j}$, $\dim V = j$] \label{ex:5-i-dimensional in 16-2^i}
Let $0 \leq j \leq 4$ be an integer.
Let $S_{16}$ and $V_{16}$ be as in Example \ref{ex:5-dimensional in 16}.
Take linearly independent elements $\phi_1, \dots, \phi_j$ in the dual vector space of $S_{16}$,
and let $H_i = (\phi_i = 0) \subset S_{16}$ be the corresponding hyperplanes.
The subspace of $V_{16}$ generated by $H_1, \dots, H_j$ is of dimension $j$ and realized on a subset of $S_{16}$ of cardinality $16 - 2^{4-j}$.
\end{example}

\begin{example}[$m = 24$, $\dim V = 12$] \label{ex:12-dimensional in 24}
The extended binary Golay code $V \subset 2^{24}$ is $12$-dimensional and the cardinalities of its elements are $0, 8, 12, 16, 24$.
\end{example}

To show the other inequality, we need some lemmas.

\begin{lem} \label{lem:average over affine hyperplane:bis}
Let $W$ be a finite dimensional affine space over a finite field,
and $\map{q}{W}{\setR}$ be a function with the property that
the sum $\sum_{w \in H} q(w)$, where $H \subset W$ is an affine hyperplane, is independent of $H$.
Then $q$ is a constant function.
\end{lem}
\begin{proof}
If $W = 0$ this is clear. Suppose $\dim W \geq 1$.
It suffices to prove $q(w_0) = (\card{W} - 1)^{-1} \sum_{w \neq w_0} q(w)$
for an arbitrary choice of $w_0 \in W$. 
We obtain two independent linear relations between $q(w_0)$ and $\sum_{w \neq w_0} q(w)$
by considering the sum along (i) all affine hyperplanes and (ii) all affine hyperplanes containing $w_0$.
\end{proof}

\begin{lem} \label{lem:intersection of subsets:ter}
Let $S$ be a finite set, $V \subset 2^S$ a vector subspace, and $k > 0$ an integer.
\begin{enumerate}
\item \label{item:intersection of subsets:ter:all constant}
Suppose that every $B \in V \setminus \set{\emptyset}$ satisfies $\card{B} = k$.
Then, for any sequence of linearly independent elements $B_1, \dots, B_l \in V$, 
and for any $l$-tuple $(B'_i) \in \prod_{i=1}^l \set{B_i, {B_i}^{\mathrm{c}}}$ not equal to $(B_i^{\mathrm{c}})$,
we have $\card{B'_1 \cap \dots \cap B'_l} = k/2^{l-1}$.
We also have $\card{S} \geq 2k(1 - 1/2^{\dim V})$.
\item \label{item:intersection of subsets:ter:all half}
Suppose that every $B \in V \setminus \set{\emptyset}$ satisfies $\card{B} = \card{S}/2$.
Then, for any sequence of linearly independent elements $B_1, \dots, B_l \in V$, and for any choices $B'_i \in \set{B_i, {B_i}^{\mathrm{c}}}$,
we have $\card{B'_1 \cap \dots \cap B'_l} = \card{S}/2^l$.
In particular, we have $\dim V \leq \ord_2{\card{S}}$.
\item \label{item:intersection of subsets:ter:one particular}
Suppose $C \in V$ is an element such that $C \neq \emptyset$ and 
every $B \in V \setminus \set{\emptyset, C}$ satisfies $\card{B} = k$.
Then, for any sequence of elements $B_1, \dots, B_l \in V$ such that $B_1, \dots, B_l, C$ are linearly independent,
and for any $l$-tuple $(B'_i) \in \prod_{i=1}^l \set{B_i, {B_i}^{\mathrm{c}}}$ not equal to $(B_i^{\mathrm{c}})$,
we have $\card{B'_1 \cap \dots \cap B'_l \cap C} = \card{C}/2^l$.
In particular, we have $\dim V \leq 1 + \ord_2 \card{C}$.
\item \label{item:intersection of subsets:ter:exceptions}
Suppose that $n = \card{\set{B \in V : \card{B} \neq 0, k}} > 0$
and that $C$ is an element of this set.
Then $\dim V \leq \ceil{\log_2 (n + 1)} + \ord_2 \card{C}$.
\end{enumerate}
\end{lem}
\begin{proof}
(\ref{item:intersection of subsets:ter:all constant})
The function 
$\map{q}{\prod_{i = 1}^l \set{B_i, {B_i}^{\mathrm{c}}}}{\setZ}$ defined by
\[
q(B'_1, \dots, B'_{l}) = 
\begin{cases}
\card{B'_1 \cap \dots \cap B'_l} & ((B'_i) \neq (B_i^{\mathrm{c}})), \\
\card{B'_1 \cap \dots \cap B'_l} + 2 k - \card{S} & ((B'_i) = (B_i^{\mathrm{c}})), 
\end{cases}
\]
satisfies the assumption of Lemma \ref{lem:average over affine hyperplane:bis}. 
Indeed, the sum over an affine hyperplane $H \subset \prod_{i = 1}^l \set{B_i, {B_i}^{\mathrm{c}}}$ is equal to 
\[
\begin{cases}
\card{x_1 B_1 + \dots + x_{l} B_{l}} & ((B_i^{\mathrm{c}}) \notin H)), \\
\card{(x_1 B_1 + \dots + x_{l} B_{l})^{\mathrm{c}}} + 2 k - \card{S} & ((B_i^{\mathrm{c}}) \in H))
\end{cases}
\]
for some $(x_i) \in (\setF_2)^l \setminus \set{0}^l$,
which is equal to $k$ in the either case.
Hence $q$ takes a constant value $c$, which is equal to $2^{-l} \sum q = 2^{-l}(\card{S} + 2k - \card{S}) = k/2^{l-1}$.
We also have $0 \leq \card{B_1^{\mathrm{c}} \cap \dots \cap B_l^{\mathrm{c}}} = c - (2k - \card{S}) = \card{S} - 2k(1 - 1/2^l)$.

(\ref{item:intersection of subsets:ter:all half})
This is a special case of (\ref{item:intersection of subsets:ter:all constant}) where $k = \card{S}/2$.

(\ref{item:intersection of subsets:ter:one particular})
Let $\namedmapandmapsto{\pi}{2^S}{2^C}{B}{B \cap C}$ be the natural projection.
For each $B \in V \setminus \set{\emptyset, C}$,
we have $\card{B \cap C} = (\card{B} + \card{C} - \card{B + C})/2 = \card{C}/2$.
This implies that if $B_1, \dots, B_l$ are linearly independent then $\pi(B_1), \dots, \pi(B_l)$ are linearly independent.
Hence we can apply (\ref{item:intersection of subsets:ter:all half})
to the vector space $W := \spanned{\pi(B_1), \dots, \pi(B_l)} \subset 2^C$ and its elements $\pi(B_1), \dots, \pi(B_l)$.

(\ref{item:intersection of subsets:ter:exceptions})
We use induction on $n$.
The case $n = 1$ is (\ref{item:intersection of subsets:ter:one particular}).
Suppose $n > 1$. 
Fix $C \in V$ with $\card{C} \neq 0, k$.
We consider vector hyperplanes $V' \subset V$ containing $C$ and the corresponding numbers 
$n' = \card{\set{B \in V' : \card{B} \neq 0, k}}$.
Letting $\overline{n'}$ be the average of $n'$ for all such $V'$, we have
\[
\overline{n'} - 1 = \frac{\card{V}/2 - 2}{\card{V} - 2} (n - 1) < \frac{1}{2} (n - 1),
\]
hence we can take $V'$ satisfying $n' - 1 < (1/2)(n - 1)$,
which implies $\ceil{\log_2 (n' + 1)} \leq \ceil{\log_2 (n + 1)} - 1$.
By the induction hypothesis we have 
\begin{align*}
\dim V = \dim V' + 1 &\leq \ord_2 \card{C} + \ceil{\log_2 (n' + 1)} + 1 \\
&\leq \ord_2 \card{C} + \ceil{\log_2 (n + 1)}.
\end{align*}
\end{proof}

\begin{proof}[Proof of Proposition \ref{prop:dimension of subspaces with 0 mod 4}]
As noted above, $g(m) \geq f(m)$ follows from the examples and the inequality $g(m) \leq g(m+1) \leq g(m)+1$.

Let $V \subset 2^S$ be as in the statement.
We prove the following, which, together with the inequality $g(m+1) \leq g(m) + 1$, conclude the proof.
\begin{enumerate}
\item \label{step:uniqueness for m=16}
If $m = 16$ then $\dim V \leq 5$,
and if the equality holds then $V$ is isomorphic to the one given in Example \ref{ex:5-dimensional in 16}.
\item \label{step:upper bound for m=17}
If $m = 17$, then $\dim V \leq 5$.
\item \label{step:upper bound for m<16}
If $m < 16$, then $\dim V \leq 4 - \ceil{\log_2(16 - m)}$.
\end{enumerate}

Let $n_i = \card{\set{B \in V : \card{B} = i}}$ for $i = 0, 8, 12, 16$.
We call an element of $V$ of cardinality $8$ (resp.\ $12$) to be an \emph{octad} (resp.\ a \emph{dodecad}) respectively.

We freely use the equations like $\card{B \cap B'} = (\card{B} + \card{B'} - \card{B + B'})/2 \in 2\setZ$
and $\card{B+B'} + \card{B} + \card{B'} \leq 2 \card{S}$.

(\ref{step:uniqueness for m=16})
For any distinct dodecads $D_1, D_2 \in V$, we have $\card{D_1 \cap D_2} = 8$, and hence the complements of $D_1$ and $D_2$ are disjoint.
We have $n_{12} < 3$, since the sum of $3$ distinct dodecads would have cardinality $4$.

Suppose $n_{12} = 1$ or $n_{12} = 2$. 
We have $n_{16} = 0$, since the sum of a dodecad and $S$ would have cardinality $4$.
By Lemma \ref{lem:intersection of subsets:ter}(\ref{item:intersection of subsets:ter:exceptions}) we have $\dim V \leq 4$.

Now suppose $n_{12} = 0$.
Since $n_{16} \leq 1$, we can take an element $C \in V$ such that every element of $V \setminus \set{\emptyset, C}$ is an octad.
By Lemma \ref{lem:intersection of subsets:ter}(\ref{item:intersection of subsets:ter:one particular}),
we have $\dim V \leq 5$, and if $\dim V = 5$ then $\card{C} = 16$.

Assume $\dim V = 5$.
Take $s_0 \in S$.
Then, in view of Lemma \ref{lem:intersection of subsets:ter}(\ref{item:intersection of subsets:ter:all half}),
$S$ admits a natural structure of an $\setF_2$-vector space as the dual of the subspace $\set{B \in V \mid s_0 \notin B} \subset V$,
and under this structure $V$ is as described in Example \ref{ex:5-dimensional in 16}.

(\ref{step:upper bound for m=17})
We observe that for any two distinct dodecads $D_1, D_2$ we have $\card{D_1^{\mathrm{c}} \cap D_2^{\mathrm{c}}} = 1$,
and that for any three we have $D_1^{\mathrm{c}} \cap D_2^{\mathrm{c}} \cap D_3^{\mathrm{c}} = \emptyset$.
Hence we have $\binom{n_{12}}{2} \leq 17$, hence $n_{12} \leq 6$.
If $n_{12} \geq 1$, then by Lemma \ref{lem:intersection of subsets:ter}(\ref{item:intersection of subsets:ter:exceptions}) we have $\dim V \leq 5$.
If $n_{12} = 0$, the same argument as in (\ref{step:uniqueness for m=16}) works.

(\ref{step:upper bound for m<16})
If $n_{12} = 0$, then we apply 
Lemma \ref{lem:intersection of subsets:ter}(\ref{item:intersection of subsets:ter:all constant}) with $k = 8$.
Now suppose $n_{12} > 0$ (which implies $m \geq 12$).
Since $m < 16$ we have $n_{12} = 1$.
If $n_{8} > 0$, then we have $m \geq 14$ while 
$\dim V \leq 3$ by Lemma \ref{lem:intersection of subsets:ter}(\ref{item:intersection of subsets:ter:one particular}).
If $n_8 = 0$, then $\dim V \leq 1$.
\end{proof}

\subsection{Characterizations of Kummer lattices} \label{subsec:Kummer lattices}

\begin{prop} \label{prop:characterization of Kummer lattices}
Let $(L', i)$ be one of 
$(A_1^{\oplus 16}, 2^5)$, 
$(D_4^{\oplus 4}, 2^2)$, 
$(D_8^{\oplus 2}, 2^1)$, 
$(D_{16}^{\oplus 1}, 2^1)$, or
$(E_8^{\oplus 2}, 2^0)$.
Then, up to isometry, there is a unique even lattice $L$ whose 
root sublattice is isomorphic to $L'$ and is of index $i$.

Moreover, there are no such lattices if we replace $i$ with a strictly greater integer.
\end{prop}
\begin{defn}
For each case in Proposition \ref{prop:characterization of Kummer lattices},
we  define the \emph{Kummer lattice of type $16 A_1, 4 D_4, 2 D_8, 1 D_{16}, 2 E_8$} to be the lattice $L$,
and denote it by $K(16 A_1), K(4 D_4), K(2 D_8), K(1 D_{16}), K(2 E_8)$.
\end{defn}

\begin{prop} \label{prop:Kummer lattices contain K(16 A_1)}
The Kummer lattices are exactly (the isomorphism classes of) the even lattices containing $K(16 A_1)$ as a sublattice of finite index.
\end{prop}
We summarize the properties of the Kummer lattices in Table \ref{table:Kummer lattices}.

Again, these propositions might be well-known to experts.

\begin{table}
\begin{tabular}{lccc}
\toprule
$L$ & $\spanned{\Roots{L}}$ & $[L : \spanned{\Roots{L}}]$ & $[L : K(16 A_1)]$ \\
\midrule
$K(16 A_1)$    & $A_1^{\oplus 16}$   & $2^5$ & $2^0$ \\
$K( 4 D_4)$    & $D_4^{\oplus 4}$    & $2^2$ & $2^1$ \\
$K( 2 D_8)$    & $D_8^{\oplus 2}$    & $2^1$ & $2^2$ \\
$K( 1 D_{16})$ & $D_{16}^{\oplus 1}$ & $2^1$ & $2^3$ \\
$K( 2 E_8)$    & $E_8^{\oplus 2}$    & $2^0$ & $2^3$ \\
\bottomrule
\end{tabular}
\caption{Kummer lattices and their properties}
\label{table:Kummer lattices}
\end{table}

\begin{proof}[Proof of Proposition \ref{prop:characterization of Kummer lattices}]
For $K(16 A_1)$, this is proved in Proposition \ref{prop:overlattices of A_1^m:bis}.

For $K(2 E_8)$ this is clear, since $E_8$ is unimodular.

Let $n \in \set{4, 8, 16}$ and $m = 16/n$.
We have $\dual{D_n} / D_n \cong (\setZ/2\setZ)^2$.
There exists a (set-theoretic) section $\map{f}{\dual{D_n} / D_n}{\dual{D_n}}$ of the natural surjection 
such that the values of $f(\functorspace)^2$ are $0, -1, -n/4, -n/4$.
For any nonzero $(\xi_1, \dots, \xi_m) \in L/D_n^{\oplus m} \subset (\dual{D_n}/D_n)^{\oplus m}$, 
we have $\sum_{j = 1}^m f(\xi_j)^2 \in 2 \setZ$ and $\neq -2$.
It follows that $f(\xi_j)^2 = -n/4$ for all $j$.
It is now straightforward to list possible even overlattices and the result is as stated.
\end{proof}

\begin{proof}[Proof of Proposition \ref{prop:Kummer lattices contain K(16 A_1)}]

We may assume $K(16 A_1)$ is given as in the proof of Proposition \ref{prop:overlattices of A_1^m:bis}.
Then $L$ corresponds to subgroups $W \subset 2^S$ containing $V$ and whose elements have cardinality divisible by $4$.
It follows that the intersection of any two elements of $W$ have cardinality divisible by $2$.

We claim that any nonzero $V$-coset of $W$ contains exactly one vector subspace of $S$ of dimension $2$.
First we show that any $V$-coset contains a tetrad (an element of cardinality $4$).
Take $A \in W \setminus V$. If $A$ is a dodecad take $A + S$. Suppose $A$ is an octad. 
If $A$ satisfies $\card{A \cap H} = 4$ for all affine hyperplanes, 
then $A$ and all vector hyperplanes generate a $5$-dimensional vector space in $2^S$ without using $S$, contradiction.
Hence there exists an affine hyperplane $H$ with $\card{A \cap H} = 2$ or $6$. We take $A + H + S$ or $A + H$ respectively.
Now we show that a tetrad $T \in W$ should be an affine subspace.
Take an affine subspace $U$ of dimension $2$ with $\card{T \cap U} \geq 3$. 
If $t \in T \setminus U$, there is a hyperplane $H$ containing $U$ and not containing $t$, which then satisfies $\card{T \cap H} = 3$, contradiction.
Finally, if $T \in W$ is an affine subspace of dimension $2$ and $0 \notin T$,
then $T_0 := \set{t - t' \mid t, t' \in T}$ is a vector subspace of dimension $2$ and $T \cup T_0$ is a hyperplane.
Hence $T_0 = T + (T \cup T_0)$ is the desired vector subspace.

Thus, such a subgroup of $2^S$ is characterized by the $2$-dimensional vector subspaces of $S$ belonging to that subgroup.
We describe Kummer lattices in this way, using a basis $v_1, \dots, v_4$ of $S$.
\begin{itemize}
\item For $K(4 D_4)$, take one such vector subspace, for example $\spanned{v_1, v_2}$.
\item For $K(2 D_8)$, take three of the form $\spanned{v_1, v_2}$, $\spanned{v_1, v_3}$, $\spanned{v_1, v_2 + v_3}$
for linearly independent elements $v_1, v_2, v_3$.
\item For $K(1 D_{16})$, take every vector subspace containing one particular nonzero element of $S$.
\item For $K(2 E_8)$, take every vector subspace contained in one particular hyperplane.
\end{itemize}

Roots of $L$ different from those of $A_1^{\oplus 16}$ are of the form 
$(1/2)(\pm x_1 \pm x_2 \pm x_3 \pm x_4)$
for $2$-dimensional affine subspaces $\set{x_1, x_2, x_3, x_4}$ belonging to $W$.
Thus there are $2^6 m$ such roots, where $m = 1, 3, 7, 7$ is the number of $2$-dimensional vector subspaces belonging to $W$.
It is straightforward to determine the root system.
\end{proof}

\subsection{Embeddings of Kummer lattices to negative-definite lattices} \label{subsec:embeddings of Kummer lattices to negative-definite lattices}

In this section we will prove Proposition \ref{prop:embedding of Kummer lattices:ter} and Corollary \ref{cor:embedding of Kummer lattices:bis}.

\begin{prop} \label{prop:embedding of Kummer lattices:ter}
Let $\map{\iota}{K}{N}$ is an embedding 
of a Kummer lattice $K$ into a negative-definite even lattice $N$.
Then $\iota(K)$ is orthogonal to the set $\Roots{N} \setminus \Roots{\overline{\iota(K)}}$,
where $\overline{\iota(K)}$ denotes the saturation.

In particular, if $\iota$ is a saturated embedding,
then $\Roots{\iota(K)}$ is orthogonal to $\Roots{N} \setminus \Roots{\iota(K)}$.
\end{prop}

\begin{cor} \label{cor:embedding of Kummer lattices:bis}
Suppose $\map{\iota}{K}{\Pic(\tilde{X})}$ is a saturated embedding of a Kummer lattice $K$
to the Picard lattice of a K3 surface $\tilde{X}$.
Then there exists an element $\gamma \in \pm \Weyl{\Pic(\tilde{X})}$ such that 
$\gamma(\iota(K))$ contains the classes of $16$ smooth rational curves $C_1, \dots, C_{16}$
forming a Kummer structure of the same type as $K$.
\end{cor}

\begin{proof}[Proof of Proposition \ref{prop:embedding of Kummer lattices:ter}] 
Let $V \subset 2^S$ be the subgroup corresponding to $K \supset A_1^{\oplus 16}$, where $\card{S} = 16$,
and let $e_v \in A_1^{\oplus 16}$ be the standard basis indexed by $v \in S$.
Take $w \in \Roots{N}$ with $w \notin \overline{\iota(K)}$.
Since $A_1^{\oplus 16} \subset K$ is of finite index,
it suffices to show that $e_v \cdot w = 0$ for all $v \in S$.

Consider the set $T = \set{v \in S \mid e_v \cdot w \neq 0}$ and let $n = \card{T}$.
Let $v_1, \dots, v_n$ be the elements of $T$.
If $\abs{e_{v_i} \cdot w} \geq 2$, then $(w \pm e_{v_i})^2 \geq 0$, contradicting negative-definiteness.
Hence we have $e_{v_i} \cdot w = \pm 1$ for all $1 \leq i \leq n$ and,
by changing signs, we may assume $e_{v_i} \cdot w = 1$.
If $n \geq 4$, then $(2 w + e_{v_1} + \dots + e_{v_4})^2 = 0$, again contradicting negative-definiteness.
Suppose $1 \leq n \leq 3$.
Then there exists $H \subset V$ such that $\card{T \cap H} = 1$.
Indeed, as in the proof of Proposition \ref{prop:dimension of subspaces with 0 mod 4}, 
$S$ admits a structure of an $\setF_2$-vector space
such that $V$ contains all affine hyperplanes,
and we can cut out any point in any set of cardinality $1$, $2$, or $3$ in an $\setF_2$-vector space by an affine hyperplane.
Then $w \cdot ((1/2) \sum_{v \in H} e_v) = (1/2) \card{T \cap H} \notin \setZ$, contradiction.
We conclude that $n = 0$.
\end{proof}

\begin{proof}[Proof of Corollary \ref{cor:embedding of Kummer lattices:bis}] 
Take a class $D \in (\iota(K))^{\perp} \subset \Pic(\tilde{X})$ of positive self-intersection.
By Proposition \ref{prop:nef up to Weyl}, by twisting by an element of $\pm \Weyl{\Pic(\tilde{X})}$, we may assume $D$ is nef and effective.
Note that $\iota(K)$ is contained in $\spanned{D}^{\perp}$, which is negative-definite.

	Next, for each $v \in \Roots{K}$,
	exactly one of $\iota(v)$ and $-\iota(v)$ is represented by an effective divisor $Z_v$.
	Applying Proposition \ref{prop:-2 divisors orthogonal to a big nef} to these $Z_v$,
	we obtain that the $Z := \bigcup \Supp Z_v$ is an ADE configuration.
	We show moreover that the classes of the irreducible components $C_1, \dots, C_n$ of each $Z_v$ belong to $\iota(K)$.
	Let $S$ be the set of roots of the sublattice generated by $[C_1], \dots, [C_n]$, which is of type $A_n$, $D_n$, or $E_n$.
	We know that there are no decomposition $S = S_1 \sqcup S_2$ into two nonempty orthogonal subsets.
	By Proposition \ref{prop:embedding of Kummer lattices:ter}, 
	we have a decomposition $S = (S \cap \iota(K)) \sqcup (S \setminus \iota(K))$ into two orthogonal subsets.
	Since $S \cap \iota(K)$ is nonempty (since it contains $[Z_v]$) we obtain $S \cap \iota(K) = S$.

The equality of types follows from Proposition \ref{prop:-2 vectors}.
\end{proof}

\begin{rem}
In fact, we can show the following.
If $K$ is an even overlattice of $A_1^{\oplus m}$ with $m \leq 16$ and 
satisfies the property of Proposition \ref{prop:embedding of Kummer lattices:ter} on embeddings to negative-definite lattices,
then it is isomorphic to $0$, $E_8$, or a Kummer lattice.
Since we do not need this result in this paper, we omit the proof.
\end{rem}

\subsection{Embeddings of Kummer lattices to Picard lattices of supersingular K3 surfaces} \label{subsec:embeddings of Kummer lattices to supersingular K3 lattices}

We determine which Kummer lattices can be embedded to Picard lattices of supersingular K3 surfaces of which Artin invariant,
and when possible, we determine the possibilities for the orthogonal complement $Q$.

\begin{lem} \label{lem:complement is type 2}
Suppose $P$ is an even non-degenerate lattice that is $2$-elementary and of type 2,
and $K \subset P$ is a saturated non-degenerate sublattice that is also $2$-elementary and of type 2.
Then the orthogonal complement of $K$ in $P$ is also $2$-elementary and of type 2.
\end{lem}
\begin{proof}
Let $Q$ be the orthogonal complement. Clearly $Q$ is also non-degenerate. 
Thus we have chain of inclusions $K \oplus Q \subset P \subset \dual{P} \subset \dual{K} \oplus \dual{Q}$, each of finite index.
Since $Q$ is saturated in $P$, the natural map $\dual{P} \to \dual{Q}$ is surjective.
Hence, for each $z \in \dual{Q}$, there is $y \in \dual{K}$ such that $x = (y, z) \in \dual{P}$.
Then, since $K$ and $P$ are $2$-elementary, we have $2 x = (2 y, 2 z) \in P \cap (K \oplus \dual{Q}) = K \oplus Q$, hence $2 z \in Q$.
Moreover, we have $z^2 = x^2 - y^2 \in \setZ$ since $K$ and $P$ are of type 2.
\end{proof}

\begin{defn} \label{def:lattice Q}
We define lattices $Q_4$ and $Q_2$ by the following pairing with respect to bases $w_1, \dots, w_6$.

$Q_4$: $w_i^2 = -2$, $w_1 \cdot w_i = 1$ ($2 \leq i \leq 6$), other intersections $= 0$. 

$Q_2$: $w_i^2 = -2$, $w_1 \cdot w_i = 1$ ($2 \leq i \leq 5$), $w_5 \cdot w_6 = 1$, other intersections $= 0$. 
\end{defn}
The pairings on $Q_4$ and $Q_2$ are visualized in Figure \ref{fig:dual graph of Q},
in which the circle with $i$ denotes the element $w_i$.

Note that $Q_4$ (resp.\ $Q_2$) is $2$-elementary
with $\dual{Q_4}/Q_4 \cong (\setZ/2\setZ)^4$
(resp.\ $\dual{Q_2}/Q_2 \cong (\setZ/2\setZ)^2$),
of sign $(+1, -5)$, and of type 2.
By Proposition \ref{prop:uniqueness of p-elementary lattice}, 
$Q_4$ (resp.\ $Q_2$) is the unique lattice satisfying these properties.

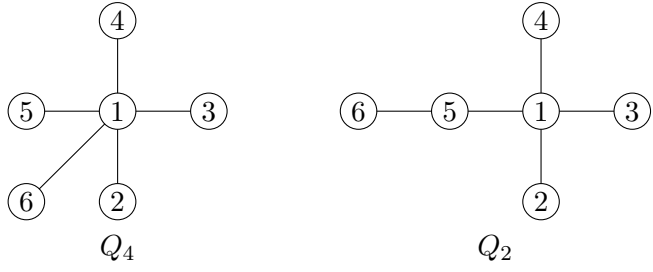
\begin{figure}
\begin{tabular}{ccc}
\begin{tikzpicture}[scale = 0.6]
	\draw (2,2) circle [radius=0.4];
	\draw (0,2) circle [radius=0.4];
	\draw (0,0) circle [radius=0.4];
	\draw (2,0) circle [radius=0.4];
	\draw (2,4) circle [radius=0.4];
	\draw (4,2) circle [radius=0.4];
	\draw (0.4,2)--(1.6,2) (2.4,2)--(3.6,2);
	\draw (2,0.4)--(2,1.6) (2,2.4)--(2,3.6);
	\draw (0.28,0.28)--(1.72,1.72);
	\draw (2,2) node {$1$};
	\draw (2,0) node {$2$};
	\draw (4,2) node {$3$};
	\draw (2,4) node {$4$};
	\draw (0,2) node {$5$};
	\draw (0,0) node {$6$};
\end{tikzpicture}
& &
\begin{tikzpicture}[scale = 0.6]
	\draw (0,2) circle [radius=0.4];
	\draw (2,2) circle [radius=0.4];
	\draw (4,0) circle [radius=0.4];
	\draw (4,2) circle [radius=0.4];
	\draw (4,4) circle [radius=0.4];
	\draw (6,2) circle [radius=0.4];
	\draw (0.4,2)--(1.6,2) (2.4,2)--(3.6,2) (4.4,2)--(5.6,2);
	\draw (4,0.4)--(4,1.6) (4,2.4)--(4,3.6);
	\draw (0,2) node {$6$};
	\draw (2,2) node {$5$};
	\draw (4,0) node {$2$};
	\draw (4,2) node {$1$};
	\draw (4,4) node {$4$};
	\draw (6,2) node {$3$};
\end{tikzpicture}
\\
$Q_4$  & \qquad \qquad & $Q_2$
\end{tabular}
\caption{Visualization of $Q_4$ and $Q_2$}
\label{fig:dual graph of Q}
\end{figure}

\begin{prop} \label{prop:structure of Q}
Let $P$ be a lattice isomorphic to the Picard lattice of a supersingular K3 surface in characteristic $2$,
and $K \subset P$ a saturated sublattice isomorphic to the Kummer lattice (of some type).
Then the orthogonal complement $Q$ of $K$ in $P$ is isomorphic to either $Q_4$ or $Q_2$.
\end{prop}

\begin{proof}

Since the Kummer lattice (of any type) is $2$-elementary and of type 2,
the complement $Q$ is also $2$-elementary and of type 2 by Lemma \ref{lem:complement is type 2}.
Write $\dual{Q}/Q \cong (\setZ/2\setZ)^b$.
By the uniqueness arguments given above,
it remains to show that $b \in \set{2, 4}$.
Clearly we have $b \leq \rank Q = 6$.

Since $\card{\dual{K}/K} \in \set{2^6, 2^4, 2^2, 2^0}$ (depending on the type of the Kummer structure)
and $\card{\dual{P}/P} = 2^{2 \sigma}$ (depending on the Artin invariant $\sigma$)
are squares,
so is $\card{\dual{Q}/Q}$,
since $\card{\dual{Q}/Q} 
= \card{\dual{(K \oplus Q)}/\dual{P}} \cdot \card{\dual{P}/P} \cdot \card{P / (K \oplus Q)} \cdot \card{\dual{K}/K}^{-1}
= \card{\dual{P}/P} \cdot \card{P / (K \oplus Q)}^{2} \cdot \card{\dual{K}/K}^{-1}$.

Hence $b \in \set{0, 2, 4, 6}$.
If $b = 0$, then $Q$ is an even unimodular lattice of sign $(+1, -5)$, which does not exist.
Suppose $b = 6$.
Then $Q(1/2)$ is a unimodular lattice of sign $(+1, -5)$.
By the classification of indefinite unimodular lattices (even ones and odd ones), we have $Q(1/2) \cong {\mathrm{I}}_{1,5}$, i.e., 
$Q(1/2)$ has a basis $e_1, \dots, e_6$ with $e_1^2 = 1$, $e_i^2 = -1$ ($2 \leq i \leq 6$), and $e_i \cdot e_j = 0$ ($i \neq j$).
Then, considering $e_i$ as elements of $Q$, we have $e_i^2 = \pm 2$,
but since $e_i \in Q = 2 \dual{Q}$ this contradicts the fact that $Q$ is of type 2.
\end{proof}

\begin{prop} \label{prop:embedding of Kummer lattices to supersingular K3 lattices}
Let $P$ be the lattice isomorphic to the Picard lattice of a supersingular K3 surface in characteristic $2$ of Artin invariant $\sigma$.
Then there exists a saturated embedding of the Kummer lattice of type $16 A_1, 4 D_4, 2 D_8, 1 D_{16}, 2 E_8$ into $P$
if and only if $\sigma \leq 5, 4, 3, 2, 2$ respectively.
\end{prop}

\begin{lem} \label{lem:glueing lattices}
Let $L_1$ and $L_2$ be even non-degenerate lattices. 
Denote by $\map{q_i}{\dual{L_i}/L_i}{\setQ/2\setZ}$ be the induced quadratic maps.
We have a canonical bijection between the sets of the following objects.
\begin{itemize}
\item Even overlattices $L$ of $L_1 \oplus L_2$ of finite index in which both $L_1$ and $L_2$ are saturated.
\item Triples $(M_1, M_2, \psi)$, where $M_i$ are subgroups of $\dual{L_i}/L_i$ respectively
and $\map{\psi}{M_1}{M_2}$ is an isomorphism of abelian groups satisfying $q_1(x) + q_2(\psi(x)) = 0$ for every $x \in M_1$.
\end{itemize}
Moreover, $L / (L_1 \oplus L_2)$ is naturally isomorphic to $M_1$ and $M_2$.
\end{lem}
\begin{proof}
Suppose we are given $L$. 
Let $M = L / (L_1 \oplus L_2)$.
The map $L/L_2 \to \dual{L_1}$ is injective since $L_2$ is saturated,
and hence $\map{\pr_1}{M}{\dual{L_1}/L_1}$ is injective.
The same holds for $\pr_2$.
Let $M_i := \Image(\pr_i)$ and $\psi := \pr_2 \circ \pr_1^{-1}$.

Given a triple $(M_1, M_2, \psi)$, let
$L = \set{(x_1, x_2) \in \dual{L_1} \oplus \dual{L_2} \mid 
(x_i \bmod L_i) \in M_i, \psi((x_1 \bmod L_1)) = (x_2 \bmod L_2)}$.
\end{proof}

\begin{proof}[Proof of Proposition \ref{prop:embedding of Kummer lattices to supersingular K3 lattices}]
Suppose $K \subset P$ is a saturated sublattice isomorphic to the Kummer lattice.
Then we have $\dual{K} / K \cong (\setZ/2\setZ)^{a}$, $a = 6, 4, 2, 0, 0$ respectively,
and $\dual{Q}/Q \cong (\setZ/2\setZ)^b$, $b \in \set{4, 2}$ by Proposition \ref{prop:structure of Q}.
Hence $2 \sigma = \log_2 \card{\dual{\Pic({\tilde{X}})}/\Pic({\tilde{X}})}
\leq \log_2 \card{\dual{K}/K} + \log_2 \card{\dual{Q}/Q}
= a + b \leq a + 4$.

Now consider the converse statement.
By the uniqueness (Proposition \ref{prop:uniqueness of p-elementary lattice}),
it suffices to give, for each pair of $T \in \set{16 A_1, 4 D_4, 2 D_8, 1 D_{16}, 2 E_8}$ and $\sigma$,
a saturated embedding of $K(T)$ into a lattice $P$ satisfying the properties characterizing $\Pic(\tilde{X})$ of Artin invariant $\sigma$..

If $T$ is $1 D_{16}$ or $2 E_8$, then the lattices $K \oplus Q_4$ and $K \oplus Q_2$ satisfy the required properties with $\sigma = 2$ and $\sigma = 1$ respectively.

Suppose $T$ is $16 A_1$, $4 D_4$, or $2 D_8$, and let $n= 4, 3, 2$ respectively.
We give subgroups $M_1 \subset \dual{K}/K$ and $M_2 \subset \dual{Q_4}/Q_4$ of order $2^n$
and an isomorphism $\map{\psi}{M_1}{M_2}$ satisfying the condition of Lemma \ref{lem:glueing lattices}.
This induces an even overlattice of $K \oplus Q_4$ of index $2^n$, which is the lattice with $\sigma = 1$, in which $K$ is saturated (and so is $Q_4$).
Replacing $M_i$ with subgroups, we obtain lattices with every $\sigma$ satisfying  $1 \leq \sigma \leq 1 + n$.

We may assume $K(16 A_1)$ is as in the proof of Proposition \ref{prop:dimension of subspaces with 0 mod 4},
and $K(16 A_1) \subset K(4 D_4) \subset K(2 D_8)$ are given as overlattices of $K(16 A_1)$ 
as in the proof of Proposition \ref{prop:Kummer lattices contain K(16 A_1)}.
Thus we describe classes of $\dual{K}/K \subset ((1/2) A_1^{\oplus 16})/K$ by 
subsets of $2^S$, using a basis $v_1, v_2, v_3, v_4$ of $S$. Let
\begin{gather*}
t_1 = \spanned{v_1, v_4}, t_2 = \spanned{v_2, v_3} \in \dual{K(2 D_8)}, \\ 
t_3 = \spanned{v_1 + v_2, v_3 + v_4} \in \dual{K(4 D_4)}, \\
t_4 = \spanned{v_1 + v_2 + v_3, v_1 + v_3 + v_4} \in \dual{K(16 A_1)}.
\end{gather*}
For $1 \leq i \leq 5$, let $u_i = (1/2) \sum_{2 \leq j \leq 6, j \neq i + 1} w_i \in \dual{Q_4}/Q_4$.
We observe that the value of the quadratic map $q$ (see Lemma \ref{lem:glueing lattices})
of the sum of $m$ distinct elements of $t_1, \dots, t_4$ (resp.\ $u_1, \dots, u_5$)
is $0, 0, 1, 1, 0$ (in $\setQ/2\setZ$) if $m = 0, 1, 2, 3, 4$ respectively.
Hence the map $\spanned{t_1, \dots, t_n} \to \spanned{u_1, \dots, u_n}$ defined by $t_i \mapsto u_i$ satisfies the required conditions.
\end{proof}

\begin{rem}
When $T$ is $16 A_1$, $4 D_4$, or $2 D_8$,
we can give similar overlattices of $K \oplus Q_2$ of index $2^n$, where $n = 2, 2, 1$.
For example, we let 
\begin{gather*}
t_1 = \set{v_1, v_2, v_3, v_2 + v_3, v_4, v_1 + v_4} \in \dual{K(2 D_8)}, \\
t_2 = \set{v_1, v_2, v_4, v_2 + v_4, v_3 + v_4, v_1 + v_3 + v_4} \in \dual{K(4 D_4)},
\end{gather*}
$u_i = (1/2) (w_i + w_3) \in \dual{Q_2}/Q_2$ ($i = 1, 2$), and use the map $\spanned{t_1, \dots, t_n} \to \spanned{u_1, \dots, u_n}$ defined by $t_i \mapsto u_i$.

Thus we obtain the following refinement of Proposition \ref{prop:embedding of Kummer lattices to supersingular K3 lattices}. 
Suppose $T$ and $\sigma$ satisfy the condition of Proposition \ref{prop:embedding of Kummer lattices to supersingular K3 lattices}.
\begin{itemize}
\item If $(T, \sigma) = (16 A_1, 5), (16 A_1, 1), (4 D_4, 4), (2 D_8, 3), (1 D_{16}, 2), (2 E_8, 2)$, then $Q$ is isomorphic to $Q_4$.
\item If $(T, \sigma) = (1 D_{16}, 1), (2 E_8, 1)$, then $Q$ is isomorphic to $Q_2$.
\item In all other cases, both $Q_4$ and $Q_2$ can occur as the orthogonal complement $Q$ of $K$.
\end{itemize}
This result is not used in this paper.
\end{rem}

\section{Kummer structure on K3 surfaces} \label{sec:Kummer structures}

\subsection{Definition of Kummer structures} \label{subsec:definition of Kummer structures}

\begin{defn} \label{def:Kummer structure}
Let $T$ be one of $16 A_1$, $4 D_4$, $2 D_8$, $1 D_{16}$, or $2 E_8$.
We say that $16$ smooth rational curves $C_1, \dots, C_{16}$ on a K3 surface $\tilde{X}$
form a \emph{Kummer structure of type $T$}
if the saturation of the sublattice of $\Pic(\tilde{X})$ generated by the classes $[C_i]$ is isomorphic to
the Kummer lattice $K(T)$ of type $T$.

We also say that Kummer structures of type other than $16 A_1$ to be of \emph{additive type}.
\end{defn}

\begin{rem}
If $C_i$ form a Kummer structure of type $T$, then, by Proposition \ref{prop:-2 vectors},
the curves $C_i$ are precisely the smooth rational curves whose classes belong to the saturation,
and their configuration is necessarily of type $T$.
\end{rem}

\begin{defn} \label{def:inseparable Kummer surface}
An \emph{inseparable Kummer surface} is a K3 surface in characteristic $2$ equipped with a Kummer structure.
We say that the contraction of the $16$ curves is a \emph{contracted inseparable Kummer surface}.
\end{defn}

\subsection{Kummer structures and usual Kummer surfaces} \label{subsec:usual Kummer surfaces}

Propositions \ref{prop:abelian surface to kummer structure} and \ref{prop:kummer structure to abelian surface}
show that, for a K3 surface of appropriate characteristic and height,
a structure of usual Kummer surface is essentially equivalent to
$16$ curves forming a Kummer structure.

\begin{prop} \label{prop:abelian surface to kummer structure}
Let $A$ be an abelian surface over an algebraically closed field $k$.
Let $\tilde{X}$ be the minimal resolution of $X := A / \set{\pm 1}$.
\begin{enumerate}
\item \label{item:abelian surface to kummer structure:16 A_1}
Suppose $\charac k \neq 2$. 
Then $\Sing(X)$ is $16 A_1$, and
the exceptional curves of $\map{\phi}{\tilde{X}}{X}$
form a Kummer structure of type $16 A_1$.
\item \label{item:abelian surface to kummer structure:additive}
Suppose $\charac k = 2$ and $\prank(A) = 2$ (resp.\ $1$).
Then the same assertions as in (\ref{item:abelian surface to kummer structure:16 A_1}) hold, with $16 A_1$ replaced with $4 D_4$ (resp. $2 D_8$).
Moreover, the height of $\tilde{X}$ is $1$ (resp.\ $2$),
and each RDP of $X$ is of type $D_4^1$ (resp.\ $D_8^2$).
\end{enumerate}
\end{prop}

\begin{proof}
In each case, let $L \subset \Pic(\tilde{X})$ be the saturation of the lattice $L'$ generated by the $16$ exceptional curves.
By Proposition \ref{prop:-2 vectors}, 
the root sublattice of $L$ is exactly $L'$.
To show that the curves form a Kummer structure,
we shall show that the singularities are as stated and that $[L : L'] \geq 2^5, 2^2, 2^1$ if $\Sing(X)$ is $16 A_1, 4 D_4, 2 D_8$ respectively.
Classes in $L/L'$ correspond to $\mu_2$-coverings (which are the same as $\setZ/2\setZ$-coverings in characteristic $\neq 2$) of $X$ that are free outside $\Sing(X)$.

(\ref{item:abelian surface to kummer structure:16 A_1})
It is clear that each singularity is $A_1$, 
and the number of singular points is equal to $\card{A[2]} = 16$.
The additional classes are given by 
the covering $A \to X$ and 
$\setZ/2\setZ$-coverings $B/\set{\pm 1} \to A/\set{\pm 1} = X$
corresponding to isogenies $B \to A$ of degree $2$.
There are $30$ such coverings in total, 
since there are $16$ choices of an origin of $A$,
there are $15$ choices of an isogeny for each choice of an origin,
and each morphism $B \to A$ appears for $8$ pairs of choices.

(\ref{item:abelian surface to kummer structure:additive})
See \cite{Katsura:Kummer 2}*{Theorem C} and \cite{Artin:wild2}*{Examples} for the assertion on quotient singularities.

The assertion on the height of $\tilde{X}$ follows from the comparison of the slopes of 
$\Hcrys^2$ of $\tilde{X}$ and those of $\Hcrys^1$ of $A$
(use Proposition \ref{prop:height of abelian surfaces} and the arguments in the proof of Proposition \ref{prop:height and Picard number}).

The additional classes are given by 
$\mu_2$-coverings $B/\set{\pm 1} \to A/\set{\pm 1} = X$
corresponding to isogenies $B \to A$ of degree $2$ that are $\mu_2$-quotients.
Since $\prank(A) = 2$ (resp.\ $1$), there are $3$ (resp.\ $1$) such coverings.
\end{proof}

\begin{rem}
If $p = 2$ and $\prank(A) = 0$ (equivalently, $A$ is supersingular), then $\Sing(A / \set{\pm 1})$ consists of one elliptic double point, and $\Km(A)$ is a rational surface.
See \cite{Shioda:Kummer2} and \cite{Katsura:Kummer2} for details.
\end{rem}

\begin{rem} \label{rem:Kummer surface is supersingular iff}
It is also known that, if $p > 0$ and $\Km(A)$ is a K3 surface, then $\Km(A)$ is supersingular if and only if $A$ is so.
This can be proved by computing the slopes of crystalline cohomology groups, as in the proof of Proposition \ref{prop:abelian surface to kummer structure}.
Thus there exist supersingular Kummer K3 surfaces in all characteristic $\geq 3$, but not in characteristic $2$.
\end{rem}

\begin{prop} \label{prop:kummer structure to abelian surface}
Let $\tilde{X}$ be a K3 surface over an algebraically closed field $k$.
Let $C_1, \dots, C_{16}$ be smooth rational curves on $\tilde{X}$
forming an ADE configuration, and let $\map{\phi}{\tilde{X}}{X}$ be the contraction of the curves.
\begin{enumerate}
\item \label{item:kummer to abelian:16 A_1}
Suppose $\charac k \neq 2$ and the configuration of the $16$ curves is $16 A_1$.
Then there is a morphism $A \to X$ that is the quotient morphism by the action of $\set{\pm 1}$ on an abelian surface $A$.
\item \label{item:kummer to abelian:4 D_4 2 D_8}
Suppose $\charac k = 2$, the configuration of the $16$ curves is $4 D_4$ (resp.\ $2 D_8$),
and the height of $\tilde{X}$ is $1$ (resp.\ $2$).
Then the same conclusion holds,
and $\prank(A) = 2$ (resp.\ $1$).
\end{enumerate}
\end{prop}

\begin{proof}
(\ref{item:kummer to abelian:16 A_1})
First we recall that
there always exists an invertible sheaf $\cL$ on $\tilde{X}$ satisfying $\cL^{\otimes 2} \cong \cO_{\tilde{X}}(-\sum_i C_i)$.
(This is no longer true in the inseparable setting considered below,
as we will see in Remark \ref{rem:divisibility condition}.)
If $k = \setC$, then this is \cite{Nikulin:Kummer surfaces}*{Theorem 1}.
If $k$ is a general field of characteristic $\neq 2$, 
then Nikulin's proof works if we replace $H^2(\tilde{X}, \setZ)$ with $\Het^2(\tilde{X}, \setZ_2)$.

Let $\cL$ be as above, that is, it satisfies $\cL^{\otimes 2} \cong \cO_{\tilde{X}}(-\sum_i C_i)$.
The contraction $\phi$ induces an isomorphism $\tilde{X} \setminus \bigcup_i C_i \isomto X^{\sm}$.
Then $\sSpec (\cO_{X^{\sm}} \oplus \cL \restrictedto{X^{\sm}}) \to X^{\sm}$ is a $\setZ/2\setZ$-torsor and extends to 
a $\setZ/2\setZ$-covering $A = \sSpec (\cO_X \oplus \phi_* \cL) \to X$,
and $A$ is an abelian surface.
Choosing one of the points above $\Sing(X)$ as the origin, this identifies $\tilde{X}$ with $\Km(A)$.
Here, it follows from the classification of the endomorphism rings of abelian varieties that 
an involution with $0$-dimensional fixed locus is the inversion.

(\ref{item:kummer to abelian:4 D_4 2 D_8})
Suppose that the configuration is $4 D_4$ and $\height(\tilde{X}) = 1$.
By Proposition \ref{prop:RDP and height}, $\Sing(X)$ is $4 D_4^1$.
For each $z \in \Sing(X)$, let $I_z \subset \cO_{X, z}$ be the ideal as in \cite{Matsumoto:k3alphap}*{Convention 6.5}
(in this case, $I_z$ is the maximal ideal $\idealm_z$), 
and let $\cI := \Ker (\cO_X \to \bigoplus_{z \in \Sing(X)} \cO_{X, z}/I_z)$.
As in \cite{Matsumoto:k3alphap}*{proof of Theorem 7.3}, we have an exact sequence
\[
 0 \to \Ext^1(\cI, \cO_X) \namedto{\beta} \bigoplus_{z \in \Sing(X)} \Ext^1(I_z, \cO_{X, z}) \namedto{\gamma} H^2(X, \cO_X) \to 0,
\]
and bases $e_z \in \Ext^1(I_z, \cO_{X, z})$ and $e_X \in H^2(X, \cO_X)$ satisfying 
$F(e_z) = e_z$ and $\gamma(e_z) = e_X$.
The extension $0 \to \cO_X \to \cB \to \cI \to 0$ corresponding to the element $e := \sum_{z \in \Sing(X)} e_z \in \Ext^1(\cI, \cO_X)$,
which is $F$-invariant,
induces an $\setZ/2\setZ$-covering $A := \sSpec \cB \to X$.
As in loc.cit., we can show that $A$ is an abelian surface.
To compute $h^1(\cO_A)$, we note that $\cO_X/\cI$ is of length $4$.
Choosing an origin, this identifies $\tilde{X}$ with $\Km(A)$.

The case where the configuration is $2 D_8$ and $\height(\tilde{X}) = 2$ is parallel.
We describe only the differences.
The singularities are $D_8^2$.
The ideal $I_z$ is not the maximal ideal $\idealm_z$ but a $\idealm_z$-primary ideal.
We have bases $e_z, f_z \in \Ext^1(I_z, \cO_{X, z})$ and $e_X \in H^2(X, \cO_X)$
such that $F(e_z) = 0$, $F(f_z) = f_z$, $\gamma(e_z) = e_X$, $\gamma(f_z) = 0$.
We use the element $e := \sum_{z \in \Sing(X)} f_z \in \Ext^1(\cI, \cO_X)$.
\end{proof}

\subsection{Characterization of supersingular K3 surfaces admitting Kummer structures} \label{subsec:characterization of supersingular Kummer}

The main subject of this paper is 
supersingular K3 surfaces in characteristic $2$
equipped with Kummer structures.
Before discussing their geometric properties, we determine which K3 surfaces admit such structures.

\begin{thm} \label{thm:characterization of supersingular Kummer}
Let $T$ be one of $16 A_1$, $4 D_4$, $2 D_8$, $1 D_{16}$, or $2 E_8$,
and let $K = K(T)$ be the Kummer lattice of type $T$.
Let $\tilde{X}$ be a supersingular K3 surface in characteristic $2$.
The following are equivalent.
\begin{enumerate}
\item \label{condi:Kummer structure} 
$\tilde{X}$ admits a Kummer structure of type $T$.
\item \label{condi:saturated embedding} 
There is a saturated embedding $K(T) \injto \Pic(\tilde{X})$.
\item \label{condi:Artin invariant}
The Artin invariant $\sigma$ of $\tilde{X}$ is $\leq 5, 4, 3, 2, 2$ respectively.
\end{enumerate}
\end{thm}

\begin{proof}
(\ref{condi:Kummer structure}) $\implies$ (\ref{condi:saturated embedding}): Trivial.

(\ref{condi:saturated embedding}) $\implies$ (\ref{condi:Kummer structure}):
This follows from Corollary \ref{cor:embedding of Kummer lattices:bis}.

(\ref{condi:saturated embedding}) $\iff$ (\ref{condi:Artin invariant}):
This is Proposition \ref{prop:embedding of Kummer lattices to supersingular K3 lattices}.
\end{proof}

\begin{rem} \label{rem:moduli dimension of Kummer}
The moduli space of supersingular K3 surfaces is stratified by the Artin invariant,
and the locus with Artin invariant $\leq i$ has dimension $i-1$ for each $1 \leq i \leq 10$.
Thus we have a $4$-dimensional family of supersingular K3 surfaces equipped with Kummer structures of type $16 A_1$
(We will also give an explicit family in Theorem \ref{thm:projective equation of supersingular Kummer}).
This is, perhaps surprisingly, 
strictly larger than the dimension $3$ of the moduli space of usual Kummer surfaces (which correspond to abelian surfaces).
On the other hand, we will see in Section \ref{subsec:Kummer quotients} that the moduli space of the abelian-like coverings have dimension $3$, which is equal to the moduli space of abelian surfaces.

Similarly, there are families of dimension $3, 2, 1, 1$ of supersingular K3 surfaces equipped with Kummer structures of type $4 D_4, 2 D_8, 1 D_{16}, 2 E_8$ respectively.
\end{rem}

\subsection{Kummer structures on supersingular K3 surface in characteristic $2$} \label{subsec:supersingular Kummer}

\begin{prop} \label{prop:covering of inseparable Kummer surface}
Let $\tilde{X}$ be an inseparable Kummer surface.
If it is of type $16 A_1$, then there is a canonical $\mu_2$-covering $\map{\pi}{A}{X}$ that satisfies the following properties.
If it is of additive type, then there is a $2$-dimensional family of $\alpha_2$-coverings $\map{\pi}{A}{X}$ that satisfy the following properties.
\begin{enumerate}
\item $A$ is reduced and Gorenstein. \label{item:covering:reduced Gorenstein}
\item The dualizing sheaf $\omega_A$ is isomorphic to $\cO_A$. \label{item:covering:trivial dualizing sheaf}
\item $h^i(\cO_A) = 1, 2, 1$ for $i = 0, 1, 2$. \label{item:covering:h^i}
\item $\Sing(A) \cap \pi^{-1}(\Sing(X)) = \emptyset$. \label{item:covering:smooth above sing}
\item $\Sing(A)$ is $1$-dimensional. In particular, $A$ is not normal.
\item $A$ is a rational surface.
\end{enumerate}
\end{prop}

\begin{proof}
First we construct $A$.
First consider the case of type $16 A_1$.
By the description of the Kummer lattice of type $16 A_1$,
there exists a line bundle $\tilde{\cL}$ on $\tilde{X}$ with an isomorphism $\tilde{\cL}^{\otimes 2} \cong \cO_{\tilde{X}}(- \sum C_i)$.
Let $\cL := \phi_* \tilde{\cL}$.
We equip $\cB := \cO_X \oplus \cL$ with a structure of an $\cO_X$-algebra using the just-mentioned isomorphism, 
and obtain a $\mu_2$-covering $A := \sSpec \cB \to X$.

Now consider the case of additive type (i.e., one of $4 D_4, 2 D_8, 1 D_{16}, 2 E_8$).
The singularities of $X$ are then $4 D_4^0, 2 D_8^0, 1 D_{16}^0, 2 E_8^0$ by Proposition \ref{prop:RDP and height}.
Let $\cI \subset \cO_X$ be the ideal defined as in 
the proof of Proposition \ref{prop:kummer structure to abelian surface}(\ref{item:kummer to abelian:4 D_4 2 D_8}) (using \cite{Matsumoto:k3alphap}*{Convention 6.5}).
As in loc.cit., $\dim_k \Ext^1(\cI, \cO_X)$ is $3$-dimensional and annihilated by $F$.
We say that a class $e \in \Ext^1(\cI, \cO_X)$ is \emph{primitive}
if its restriction to $\Ext^1(\cI_z, \cO_{X,z})$ generates the $\cO_{X,z}$-module for every $z \in \Sing(X)$.
The subset of primitive classes is the complement of the union of finitely many hyperplanes.
Then the exact sequence $0 \to \cO_X \to \cB \to \cI \to 0$ corresponding to each primitive element 
induces an $\alpha_2$-covering $A := \sSpec \cB \to X$.
Proportional elements of $\Ext^1(\cI, \cO_X)$ induce the same surface with different $\alpha_p$-actions.

Now we show the properties of $A$.
(\ref{item:covering:reduced Gorenstein}), (\ref{item:covering:trivial dualizing sheaf}), (\ref{item:covering:h^i}), and (\ref{item:covering:smooth above sing})
are proved as in \cite{Matsumoto:k3alphap}*{proof of Theorem 7.3}.
For (\ref{item:covering:h^i}), we note that $\cO_X/\cI$ is of length $4$.

Let $\eta$ be the $1$-form on $X^{\sm}$ induced by the restriction to $X^{\sm}$ of the $G$-covering $A \to X$ (Proposition \ref{prop:1-form corresponding to G-torsor}).
Let $\delta \in H^0(X^{\sm}, \Theta)$ induced from $\eta$ by the isomorphism $\Omega^1_{X^{\sm}} \cong \Theta_{X^{\sm}}$ (recall that the canonical divisor is trivial).
It satisfies $\Fix(\delta) \cap \Sing(X) = \emptyset$.
By applying the Katsura--Takeda formula (Theorem \ref{thm:Katsura--Takeda}) to the extension of $\delta$ to $\tilde{X}$, 
we obtain $\divisorialfix{\delta}^2 = -24 + 32 + \deg \isolatedfix{\delta} > 0$.
Here, $24$ is $\deg c_2(X)$, and 
$32$ is the contribution from the exceptional curves (computed by using \cite{Matsumoto:k3alphap}*{Lemma 3.11}).
In particular, the divisor $\divisorialfix{\delta} \subset X^{\sm}$ is nontrivial, 
and hence $\Sing(A)$ is $1$-dimensional.
Therefore the covering $A$ is always non-normal.

The normalization $\normalization{A}$ of $A$ 
has anti-effective canonical divisor and
satisfies $\thpower{(\normalization{A})}{2} = X^{\delta}$ (\cite{Matsumoto:k3alphap}*{Proposition 2.15(4)}).
By \cite{Matsumoto:k3alphap}*{Proposition 4.1}, $\normalization{A}$ is a rational surface.
\end{proof}

\begin{rem}
We note that this proposition holds under a weaker assumption on the $16$ curves:
The $\mu_2$-covering can be constructed 
if the configuration is $16 A_1$ and 
$\sum [C_i] \in \Pic(\tilde{X})$ is divisible by $2$,
and the family of $\alpha_2$-coverings can be constructed
if the configuration of $C_i$ is $4 D_4$, $2 D_8$, $1 D_{16}$, or $2 E_8$, without assuming any divisibility condition.
The divisibility condition will be used in the subsequent sections.
\end{rem}

\begin{rem} \label{rem:divisibility condition}
As we recalled in the proof of Proposition \ref{prop:kummer structure to abelian surface},
if $C_1, \dots, C_{16}$ are disjoint smooth rational curves on a K3 surface in characteristic $\neq 2$,
then $[C_1] + \dots + [C_{16}] \in \Pic(\tilde{X})$ is automatically divisible by $2$.
We show that this does not hold in characteristic $2$.

Let $H \in H^0(\bP^2, \cO(6))$ be a generic element. Then the double sextic surface $X = (w^2 = H(x, y, z))$
is a supersingular RDP K3 surface with $\Sing(X) = 21 A_1$ and Artin invariant $\sigma = 10$,
and $\Pic(\tilde{X})$ is generated by $e_1, \dots, e_{21}, h$, and $(1/2)(h + \sum e_i)$,
where $e_i$ are the classes of the exceptional curves and $h$ is the hyperplane class.
It follows that, while the $21$ exceptional curves are disjoint, no nontrivial sum of their classes is divisible by $2$. 
\end{rem}

\section{Spaces of $1$-forms on RDP K3 surfaces} \label{sec:dimension of Z_infty}

In the proof of Theorem \ref{thm:projective equation of supersingular Kummer},
we will need the following upper bound for contracted inseparable Kummer surfaces $X$.
We consider the sheaves $B_i \Omega^1 \subset B_{\infty} \Omega^1 \subset Z_{\infty} \Omega^1 \subset Z_i \Omega^1 \subset \Omega^1$ on $X^{\sm}$
(see Section \ref{subsec:Cartier operator} for their definitions).
\begin{thm} \label{thm:characterization of Kummer types using Z_infty}
Let $X$ be a supersingular RDP K3 surface in characteristic $2$,
with at most $16$ exceptional curves.
Then $\dim H^0(X^{\sm}, Z_{\infty} \Omega^1) \leq 5$, and the equality holds if and only if the exceptional curves form a Kummer structure (of any type),
that is, $X$ is a contracted inseparable Kummer surface.
\end{thm}

To prove this, we study $\dim H^0(X^{\sm}, B_n \Omega^1)$ and $\dim H^0(X^{\sm}, Z_{\infty} \Omega^1)/H^0(X^{\sm}, B_{\infty} \Omega^1)$ 
for general RDP K3 surfaces $X$.
We determine $\dim H^0(X^{\sm}, B_n \Omega^1)$ (Corollary \ref{cor:dim of B_n of RDP K3}(\ref{item:dim of B_infty of RDP K3})).
This is based on local computations on RDPs explained in Section \ref{subsec:B_n of RDP}, 
which are essentially done in \cite{Liedtke--Martin--Matsumoto:RDPtors}.
We give upper bounds for $\dim H^0(X^{\sm}, Z_{\infty} \Omega^1)/H^0(X^{\sm}, B_{\infty} \Omega^1)$ (Lemma \ref{lem:dimension of C=1 in terms of RDPs}). 
We then give a combinatorial argument to show the inequality of Theorem \ref{thm:characterization of Kummer types using Z_infty}.

\subsection{Witt vectors}
For $\setF_p$-algebras $A$ and non-negative integers $n$, we have the ring $W_n A$ of (truncated) Witt vectors, 
together with ring homomorphisms $\map{R}{W_{n+1} A}{W_n A}$ called restriction,
ring homomorphisms $\map{F}{W_n A}{W_n A}$ called Frobenius, 
abelian group homomorphisms $\map{V}{W_n A}{W_{n+1} A}$ called Verschiebung,
satisfying certain relations,
and canonical isomorphisms $W_1(A) \cong A$ and $W_0(A) = 0$.
These are all functorial on $A$.
We do not recall the description of the ring structures and the morphisms, 
since all computations needed are done in our previous papers \cite{Matsumoto:k3rdpht} and \cite{Liedtke--Martin--Matsumoto:RDPtors}.

For $\setF_p$-schemes $Y$ and non-negative integers $n$, we have the schemes $W_n Y$
with underlying topological space homeomorphic to $Y$ if $n > 0$ and $\emptyset$ if $n = 0$, and with structure sheaf $W_n \cO_Y$.

\subsection{$B_n$ for RDPs} \label{subsec:B_n of RDP}

Let $Y$ be the spectrum of a local ring at an RDP $z$ of an algebraic surface,
and $\tilde{Y} \to Y$ its minimal resolution.

\begin{prop} \label{prop:isomorphisms of B_n:RDP}
Let $n \geq 0$ an integer.
\begin{enumerate}
\item \label{item:isomorphisms of B_n:RDP}
We have isomorphisms 
\[
H^0(Y^{\sm}, B_n \Omega^1) / H^0(\tilde{Y}, B_n \Omega^1) \isomto
H^1(Y^{\sm}, W_n \cO)[F] \isomto H^2_z(W_n \cO)[F].
\]
Here, 
$[F]$ denotes the kernel of the action of the Frobenius, 
and $H^2_z$ is the local cohomology.
\item \label{item:isomorphisms of B_n:I and V}
The isomorphisms in (\ref{item:isomorphisms of B_n:RDP}) are compatible with 
the inclusion $B_{n-1} \injto B_n$ and the Verschiebung $\map{V}{W_{n-1}}{W_n}$.
\item \label{item:isomorphisms of B_n:C and R}
The isomorphisms in (\ref{item:isomorphisms of B_n:RDP}) are compatible with 
the Cartier operator $\map{C}{B_n}{B_{n-1}}$ and the restriction map $\map{R}{W_n}{W_{n-1}}$.
\end{enumerate}
\end{prop}
\begin{proof}
Over smooth varieties (in characteristic $p$), we have short exact sequences
\[
0 \to W_n \cO \to F_* W_n \cO \to B_n \Omega^1 \to 0
\]
of quasi-coherent sheaves (of $W_n \cO$-modules), given by Serre \cite{Serre:topologie}*{Proposition 8}, 
where the map on the right is $(f_0, f_1, \dots, f_{n-1}) \mapsto \sum_{i = 0}^{n-1} f_i^{p^{n-1-i}-1} df_i$.
These sequences satisfy the compatibility conditions as in 
(\ref{item:isomorphisms of B_n:I and V}) and (\ref{item:isomorphisms of B_n:C and R}).
Since $z$ is a rational singularity, 
the cohomology groups $H^i(\tilde{Y}, \functorspace)$ of $W_n \cO$ and $F_* W_n \cO$ 
coincide with $H^i(Y, \functorspace)$, and vanish for $i > 0$ since $W_n Y$ is affine.
Thus, we obtain a commutative diagram with exact columns
\[ 
\begin{tikzcd}
H^0(Y, W_n \cO) \ar[d] \ar[r, "\sim"] & 
H^0(Y^{\sm}, W_n \cO) \ar[d] \\
H^0(Y, F_* W_n \cO) \ar[d] \ar[r, "\sim"] & 
H^0(Y^{\sm}, F_* W_n \cO) \ar[d] \\
H^0(\tilde{Y}, B_n \Omega^1) \ar[d] \ar[r] &
H^0(Y^{\sm}, B_n \Omega^1) \ar[d] \\
0 \ar[d] \ar[r] &
H^1(Y^{\sm}, W_n \cO) \ar[d,"F"] \\
0 \ar[r] &
H^1(Y^{\sm}, F_* W_n \cO).  
\end{tikzcd}
\]
The two horizontal morphisms on the top are isomorphisms since the sheaves $W_n \cO$ and $F_* W_n \cO$ satisfy the condition $(S_2)$.
Hence we obtain an isomorphism $H^0(Y^{\sm}, B_n \Omega^1) / H^0(\tilde{Y}, B_n \Omega^1) \isomto H^1(Y^{\sm}, W_n \cO)[F]$.

The second isomorphism comes from the long exact sequence 
\[
\dots \to H^1(Y, W_n \cO) \to H^1(Y^{\sm}, W_n \cO) \to H^2_z(W_n \cO) \to H^2(Y, W_n \cO) \to \dots,
\]
where $H^i(Y, W_n \cO) = 0$ ($i = 1, 2$) since $W_n Y$ is an affine scheme.
\end{proof}

Let $z$ and $Y$ be as above.
We write $\bar{B}_{n, z} := H^0(Y^{\sm}, B_n \Omega^1) / H^0(\tilde{Y}, B_n \Omega^1)$.
It turns out (Proposition \ref{prop:dim B}) that
$\dim_k \bar{B}_{n, z}$ is finite, that it depends only on $n$ and the type of the RDP,
and that it is constant for $n$ large enough.
We write $\bar{B}_{\infty, z} := \bar{B}_{n, z}$ for $n$ large enough.
\begin{defn}
We define the \emph{$B$-index} of $z$, and denote it by $n_B(z)$, 
to be the smallest non-negative integer $n$ such that 
$\dim \bar{B}_{n, z} = \dim \bar{B}_{n+1, z} = \dots$ holds.
\end{defn}
In particular, the $B$-index of $z$ is $0$ if and only if $\dim \bar{B}_{n, z} = 0$ for all $n$.

The $B$-index and the dimensions of $\bar{B}_{n, z}$ are essentially determined 
in \cite{Liedtke--Martin--Matsumoto:RDPtors}*{Proposition 6.2}, as follows.

\begin{prop} \label{prop:dim B}
The $B$-index of $z$ is $0$ if and only if $z$ is an $F$-injective RDP.

For non-$F$-injective RDPs,
the $B$-index and the sequence $\dim_k \bar{B}_{n, z}$ ($1 \leq n \leq n_B(z)$) are as in Table \ref{table:dim B_n}.
\end{prop}
Note that we use the convention that the range of $r$ in $D_{2l+1}^r$ in characteristic $2$
is $\set{1/2, 3/2, \dots, l-1/2}$ as in \cite{Matsumoto:rdpderi}*{Convention 1.2}, 
which is also used in \cite{Liedtke--Martin--Matsumoto:RDPtors}.

\begin{table}
\centering
\begin{tabular}{llll}
\toprule
$p$ & type & $B$-index & $\dim_k \bar{B}_{n, z}$ \\
\midrule
$5$ & $E_8^0$ & $1$ & $1$ \\
\midrule
$3$ & $E_8^1, E_7^0, E_6^0$ & $1$ & $1$ \\
$3$ & $E_8^0$ & $2$ & $1, 2$ \\
\midrule
$2$ & $E_8^3, E_7^2, E_6^0$ & $1$ & $1$ \\
$2$ & $E_8^2, E_7^1$ & $2$ & $1, 2$ \\
$2$ & $E_8^1, E_7^0$ & $3$ & $1, 2, 3$ \\
$2$ & $E_8^0$ & $3$ & $2, 3, 4$ \\
$2$ & $D_N^r$ & $\ceil{\log_2 ((N/2) - r)}$ & $\ceil{(1 - 2^{-n})((N/2) - r - 1)}$ \\ 
\bottomrule
\end{tabular}
\caption{$B$-index and the dimensions of $\bar{B}_{n, z}$}
\label{table:dim B_n}
\end{table}

\begin{proof}
We may assume that $\cO_{Y, z}$ is complete.
In \cite{Liedtke--Martin--Matsumoto:RDPtors}*{Proposition 6.2}, 
we gave generators for the larger space 
$\varinjlim_{m, n} H^2_z(W_n \cO)[F^m]$
and determined the actions of $V$ and $F$.
Restricting to the subspaces annihilated by $V^n$ and $F$, we obtain bases for the subspaces 
$H^2_z(W_n \cO)[F]$.

Let us look at the case of $D_N^r$ in characteristic $2$ in more detail.
We have a basis 
\[
\set{f_j^{(l)} \mid l, j \in \setN_{>0}, \; l \leq n, \; c(l, j) \leq (N/2) - r - 1},
\]
where the function $\map{c}{\setN_{>0} \times \setN_{>0}}{\setN_{>0}}$ is defined by 
$c(l, j) := 2^{l-1} (2 j - 1)$ and is a bijection.
The action of $V$ (to be precise, $RV$) is $V(f_j^{(l)}) = f_j^{(l-1)}$ for $l \geq 2$ and $V(f_j^{(1)}) = 0$.
It follows that $\dim_k \bar{B}_{n, z}$ is equal to the number of positive integers $1 \leq i \leq (N/2) - r - 1$
satisfying $\ord_2(i) < n$,
which is equal to $(N/2) - r - 1 - \floor{((N/2) - r - 1)/2^n} = \ceil{(1 - 2^{-n})((N/2) - r - 1)}$.
\end{proof}

\subsection{$B_n$ for RDP K3 surfaces}

Now let $X$ be an RDP K3 surface.
Recall that, for each non-negative integer $n$, 
the $W_n(k)$-module $H^i(X, W_n \cO_X)$ is of length $n, 0, n$ for $i = 0, 1, 2$,
and that, for each non-negative integers $n \leq n'$, 
\[
0 \to H^2(X, W_n \cO_X) \namedto{V^{n' - n}} H^2(X, W_{n'} \cO_X) \namedto{R^{n}} H^2(X, W_{n' - n} \cO_X) \to 0
\]
is exact
(see \cite{Matsumoto:k3rdpht}*{Lemma 5.1}).

\begin{prop} \label{prop:isomorphisms of B_n:K3}
We have isomorphisms
\[ 
H^0(X^{\sm}, B_n \Omega^1) \isomto H^0(X^{\sm}, B_n \Omega^1) / H^0(\tilde{X}, B_n \Omega^1) \isomto H^1(X^{\sm}, W_n \cO_X)[F].
\]
\end{prop}
\begin{proof}
The first map is an isomorphism since $H^0(\tilde{X}, B_n \Omega^1) \subset H^0(\tilde{X}, \Omega^1) = 0$.
The second can be showed as in the proof of Proposition \ref{prop:isomorphisms of B_n:RDP},
where we use that $\tilde{X}$ is a K3 surface to show $H^1(\tilde{X}, W_n \cO_{\tilde{X}}) = 0$.
\end{proof}

For each $z \in X$ and $n \geq 0$, there is a morphism
$\map{\gamma = \gamma_{n, z}}{H^2_z(W_n \cO_X)}{H^2(X, W_n \cO_X)}$.
The morphism $H^2_z(\cO_X)[\idealm_z] \namedto{\gamma_{1, z}} H^2(X, \cO_X)$ is an isomorphism
(see \cite{Matsumoto:k3rdpht}*{proof of Proposition 6.5}).

\begin{prop} \label{prop:surjective to H^2 of K3}
Let $z \in \Sing(X)$.
For any non-negative integer $n \leq n_B(z)$,
the map $H^2_z(W_n \cO_X)[F] \namedto{\gamma_{n, z}} H^2(X, W_n \cO_X)$ is surjective.
In particular, the Frobenius action on $H^2(X, W_n \cO_X)$ is $0$.
\end{prop}
\begin{proof}
The assertion is trivial if $n = 0$. Hereafter we assume $n > 0$ and hence $n_B(z) > 0$.

Since $\map{R}{H^2(X, W_n \cO_X)}{H^2(X, W_{n-1} \cO_X)}$ is surjective,
it suffices to show the case $n = n_B(z)$.
Moreover, since $H^2(X, W_n \cO_X) / \Ker(R^{n-1}) \cong H^2(X, \cO_X)$ is $1$-dimensional over $k$,
it suffices to find an element $\xi \in H^2_z(W_n \cO_X)[F]$ with $\gamma(\xi) \notin \Ker(R^{n-1})$.

Let $\xi$ be any element of the complement of $H^2_z(W_{n-1} \cO_X)[F]$ in $H^2_z(W_n \cO_X)[F]$,
which is nonempty by the definition of $n_B(z)$.
The element $R^{n-1}(\xi) \in H^2_z(\cO_X)$ is nonzero.
If $R^{n-1}(\xi) \in H^2_z(\cO_X)[\idealm_z] \setminus \set{0}$ then, 
since $H^2_z(\cO_X)[\idealm_z] \namedto{\gamma_{1, z}} H^2(X, \cO_X)$ is an isomorphism,
the element $\gamma(R^{n-1}(\xi)) = R^{n-1}(\gamma(\xi))$ is nonzero, as desired.
In the general case, we can find an element $f \in \idealm_z \subset \cO_{X, z}$ 
such that $f R^{n-1}(\xi) \in H^2_z(\cO_X)[\idealm_z] \setminus \set{0}$, 
and then we use $[f] \cdot \xi$ in place of $\xi$, where $[f] \in W_n \cO_{X, z}$ is the Teichm\"uller lift of $f$.
(Actually, it follows from the explicit results of Proposition \ref{prop:dim B} that 
$\dim \bar{B}_{n(B), z} / \bar{B}_{n(B)-1, z} = 1$ and hence $R^{n-1}(\xi) \in H^2_z(\cO_X)[\idealm_z] \setminus \set{0}$ for any choice of $\xi$.)
\end{proof}

\begin{defn}
We define the \emph{$B$-index} of an RDP K3 surface as 
$n_B(X) := \max \set{n_B(z) \mid z \in \Sing(X)}$ if $\Sing(X) \neq \emptyset$, 
and $n_B(X) := 0$ if $\Sing(X) = \emptyset$.
\end{defn}

\begin{prop} \label{prop:exact sequence of B_n}
Let $X$ be an RDP K3 surface. 
For each non-negative integer $n$, we have an exact sequence
\[
0 \to H^0(X^{\sm}, B_n \Omega^1)
\to \bigoplus_{z \in \Sing(X)} \bar{B}_{n, z} 
\namedto{\gamma} H^2(X, W_n \cO_X),
\]
which satisfies the compatibilities as in Proposition \ref{prop:isomorphisms of B_n:RDP}.
Moreover, $\gamma$ is surjective if $n \leq n_B(X)$.
\end{prop}
\begin{proof}
Let $n$ be an arbitrary non-negative integer for the moment,
and let $S \subset X$ be a finite set of closed points.
Let $\cI \subset W_n \cO_X$ be a sheaf of ideals with $\Supp(W_n \cO_X/\cI) \subset S$.
Mimicking the arguments of \cite{Matsumoto:k3alphap}*{proof of Lemma 6.8},
we obtain an exact sequence
\[
0 \to \Ext^1_X(\cI, W_n \cO_X) \to \bigoplus_{z \in S} \Ext^1_{X_z}(\cI_z, W_n \cO_{X, z}) \to H^2(X, W_n \cO_X),
\]
where $X_z := \Spec \cO_{X, z}$.
By taking the inductive limit with respect to $\cI$, we obtain an exact sequence
\[
0 \to H^1(X \setminus S, W_n \cO_X) \to \bigoplus_{z \in S} H^1(X_z \setminus \set{z}, W_n \cO_{X, z}) \to H^2(X, W_n \cO_X).
\]

Letting $S = \Sing(X)$, and taking the Frobenius kernels, we obtain an exact sequence
\[
0 \to H^1(X^{\sm}, W_n \cO_X)[F] \to \bigoplus_{z \in \Sing(X)} H^1(X_z^{\sm}, W_n \cO_{X, z})[F] \namedto{\gamma} H^2(X, W_n \cO_X)[F].
\]
By Propositions \ref{prop:isomorphisms of B_n:RDP} and \ref{prop:isomorphisms of B_n:K3} this is the desired exact sequence, 
and if $n \leq n_B(X)$ then $\gamma$ is surjective by Proposition \ref{prop:surjective to H^2 of K3}.
\end{proof}

\begin{cor} \label{cor:dim of B_n of RDP K3}
Let $X$ be an RDP K3 surface. 
\begin{enumerate}
\item \label{item:B_n of RDP K3 is constant after B-index}
The sequence $H^0(X^{\sm}, B_n \Omega^1)$ is constant for $n \geq n_B(X)$.
\item \label{item:dim of B_n of RDP K3}
We have $\dim_k H^0(X^{\sm}, B_n \Omega^1) = \sum_{z \in \Sing(X)} \dim_k \bar{B}_{n, z} - n$
for $0 \leq n \leq n_B(X)$.
\item \label{item:dim of B_infty of RDP K3}
We have $\dim_k H^0(X^{\sm}, B_{\infty} \Omega^1) = \sum_{z \in \Sing(X)} \dim_k \bar{B}_{\infty, z} - n_B(X)$.
\end{enumerate}
In particular, $\dim_k H^0(X^{\sm}, B_n \Omega^1)$ and $\dim_k H^0(X^{\sm}, B_{\infty} \Omega^1)$ depend only on the types of the RDPs.
\end{cor}
\begin{proof}
(\ref{item:B_n of RDP K3 is constant after B-index})
In the exact sequences of Proposition \ref{prop:exact sequence of B_n} for varying $n \geq n_B(X)$,
the transition maps at the middle (resp.\ right) are constant (resp.\ injective).

(\ref{item:dim of B_n of RDP K3})
Dimension counting.
(Under the assumption $n \leq n_B(X)$, the action of $p = FV$ on $H^2(X, W_n \cO_X)$ is $0$,
which implies that it is a $k$-vector space, not merely a $W_n(k)$-module.)

(\ref{item:dim of B_infty of RDP K3})
This follows from (\ref{item:B_n of RDP K3 is constant after B-index}) and (\ref{item:dim of B_n of RDP K3}).
\end{proof}

\subsection{Proof of Theorem \ref{thm:characterization of Kummer types using Z_infty}}

For each RDP $z$, we define 
$i_z$ to be the index of the type of the RDP (i.e.\ $i_z = N$ if $z$ is of type $A_N$, $D_N$, or $E_N$).
We also define $m_z, b_z \in \setN$ by 
\[
m_z = 
\begin{cases}
l+1 & (A_{2l+1}), \\
0 & (A_{2l}), \\
l+1 & (D_{2l}), \\
2 & (D_{2l+1}), \\
0 & (E_6), \\
3 & (E_7), \\
0 & (E_8), 
\end{cases}
\qquad
b_z = \dim \bar{B}_{\infty, z} = 
\begin{cases}
0 & (A_N), \\
l-1 & (D_{2l}^0), \\
l-1 & (D_{2l+1}^{1/2}), \\
1 & (E_6^0), \\
3 & (E_7^0), \\
4 & (E_8^0).
\end{cases}
\]
We recall the function $\map{f}{\set{0, 1, \dots, 21}}{\setN}$ introduced in Section \ref{subsec:overlattices of A_1^m}:
$f(m) = 4 - \ceil{\log_2 (16-m)}$ for $0 \leq m < 16$,
$f(16) = 5$, and 
$f(m) = m - 12$ for $17 \leq m \leq 21$.

\begin{lem} \label{lem:dimension of C=1 in terms of RDPs}
If $X$ is an RDP K3 surface in characteristic $2$,
then we have $\dim_k H^0(X^{\sm}, Z_{\infty})/H^0(X^{\sm}, B_{\infty}) \leq f(m)$,
where $m = \sum_{z \in \Sing(X)} m_z$.
\end{lem}
\begin{proof}
By Proposition \ref{prop:torsors and 1-forms}, the dimension in question is equal to $\log_2 \card{(L/L')[2]}$,
where $L$ is the saturation of the lattice $L' \subset \Pic(\tilde{X})$ generated by the exceptional curves.

For each RDP $z$, let $L'_z$ be the lattice generated by the exceptional curves above $z$ (which is isomorphic to the lattice $A_N$, etc., if $z$ is of type $A_N$, etc.).
There is a set of $m_z$ disjoint exceptional curves 
such that every element of $(\dual{(L'_z)}/L'_z)[2]$ is represented by $(1/2)$ times a sum of (not necessarily all of) those curves.
Hence $(L/L')[2]$ can be embedded into $\dual{(A_1^{\oplus m})}/(A_1^{\oplus m})$.
By applying Proposition \ref{prop:overlattices of A_1^m:bis} to the image,
which satisfies the condition on the cardinalities since $L$ contains no more roots than $L'$,
we obtain the desired bound.
\end{proof}

We now prove the following.
\begin{prop} \label{prop:leq 5 for configuration of index 16}
Consider a finite collection (possibly with repetitions) of types of RDPs in characteristic $2$ 
belonging to $\set{A_N, D_{2l}^0, D_{2l+1}^{1/2}, E_6^0, E_7^0, E_8^0}$.
Let $i = \sum_z i_z$, $m = \sum_z m_z$, $b = \sum_z b_z$, and $n_B = \max_z n_B(z)$.
If $i \leq 16$, then $f(m) + b - n_B \leq 5$,
and the equality holds if and only if the collection is one of $16 A_1, 4 D_4^0, 2 D_8^0, 1 D_{16}^0, 2 E_8^0$.
\end{prop}
Theorem \ref{thm:characterization of Kummer types using Z_infty} follows from this. 
Indeed, for $X$ as in Theorem \ref{thm:characterization of Kummer types using Z_infty}, we have
$\dim H^0(X^{\sm}, Z_{\infty} \Omega^1)/H^0(X^{\sm}, B_{\infty} \Omega^1) \leq f(m)$ by Lemma \ref{lem:dimension of C=1 in terms of RDPs}
and $\dim H^0(X^{\sm}, B_{\infty} \Omega^1) = b - n_B$ by Corollary \ref{cor:dim of B_n of RDP K3}(\ref{item:dim of B_infty of RDP K3}).
Here we use Proposition \ref{prop:RDP and height} to determine the coindex of $D_N$ and $E_N$.
Hence $\dim H^0(X^{\sm}, Z_{\infty} \Omega^1) \leq 5$ by Proposition \ref{prop:leq 5 for configuration of index 16}.
The equality holds if and only if the exceptional curves form a Kummer structure.

\begin{proof}[Proof of Proposition \ref{prop:leq 5 for configuration of index 16}]
By adding $A_1$'s if necessary, we may assume $i = 16$. 

The procedure replacing $A_N$ ($N \geq 2$) with $N A_1$,
$D_{2l+1}^{1/2}$ with $D_{2l}^0 + A_1$, or
$E_6$ with $D_4^0 + 2 A_1$ 
preserves $i, b, n_B$ and increases $m$.
The procedure replacing $E_7^0 + A_1$ with $D_8^0$ preserves $i$ and $b$, increases $m$, 
and either preserves or decreases $n_B$.
Hence, by applying these, 
we may assume that only $A_1$, $D_{2l}^0$, and $E_8^0$ appear.

At the end of the proof, we can show that if the collection became one of the equality cases after these procedures,
then the original value of $f(m) + b - n_B$ is strictly less than $5$. 
We omit the verification.

For the readers' convenience, we recall the values:
\[
(i_z, m_z, b_z) = 
\begin{cases}
(1, 1, 0) & (A_1), \\
(2l, l+1, l-1) & (D_{2l}^0), \\
(8, 0, 4) & (E_8^0),
\end{cases}
\qquad
n_B(z) = 
\begin{cases}
0 & (A_1), \\
1 & (D_4^0), \\
2 & (D_6^0, D_8^0), \\
3 & (D_{10}^0, \dots, D_{16}^0, E_8^0).
\end{cases}
\]

Suppose there is at least one $E_8^0$.
We have $m \leq 8$ and $n_B = 3$.
If $m = 8$, then the collection is $E_8^0 + 8 A_1$, in which case $f(m) + b - n_B = 2$.
Otherwise, we have $m < 8$ and hence $f(m) = 0$.
We have $b \leq (1/2) \cdot 16 = 8$ since $b_z \leq (1/2) i_z$ for all $z$,
and the equality holds only if every $z$ is $E_8^0$.
Thus we have $f(m) + b - n_B \leq 5$, with equality only if the collection is $2 E_8^0$.

Now suppose there is no $E_8^0$.
We have $i_z = m_z + b_z$ for all $z$, hence $m + b = 16$.
Suppose $n_B = 0$ (resp.\ $1$, $2$, $3$).
Let $c = 0$ (resp.\ $1/4$, $3/8$, $7/16$): this is the smallest constant such that 
$b_z \leq c i_z$ for all $z$ satisfying $n_B(z) \leq n_B$,
and the equality holds only for $A_1$ (resp.\ $D_4^0$, $D_8^0$, $D_{16}^0$).
Hence $m = 16 - b \geq 16 (1 - c) = 16$ (resp.\ $12$, $10$, $9$),
and the equality holds only if every $z$ is $A_1$ (resp.\ $D_4^0$, $D_8^0$, $D_{16}^0$).
We observe that the maximum value of $f(m) + b = f(m) - m + 16$ in this range
is $5$ (resp.\ $6$, $7$, $8$), achieved only at $m = 16$ (resp.\ $12$, $10$, $9$).
We conclude that, under the assumption on $n_B$, the inequality $f(m) + b - n_B \leq 5$ holds,
and the equality holds only if the collection is $16 A_1$ (resp.\ $4 D_4^0$, $2 D_8^0$, $1 D_{16}^0$).
\end{proof}

\begin{rem}
Lemma \ref{lem:dimension of C=1 in terms of RDPs} does not give the sharp bound for all collections of types of RDPs.
For example, suppose $\Sing(X)$ is $13 A_1 + D_4$, then $m$ is $16$ and the lemma gives an upper bound $5$.
However, any element of the corresponding space $V$ should contain an even number of the curves among the $3$ curves over the RDP of type $D_4$,
which makes dimension $5$ impossible.
\end{rem}

\section{Projective equations} \label{sec:projective equations}

Let $\tilde{X}$ together with $C_1, \dots, C_{16}$ be an inseparable Kummer surface, and $X$ the contraction of the curves.
Let $K \subset \Pic(\tilde{X})$ be the saturation of the sublattice generated by $[C_i]$, which is isomorphic to the Kummer lattice.
Let $Q$ be the orthogonal complement of $K$ in $\Pic(\tilde{X})$.
We know by Proposition \ref{prop:structure of Q} that $Q$ is isomorphic to $Q_4$ or $Q_2$ defined in Definition \ref{def:lattice Q}.
We then say that the inseparable Kummer surface (or the Kummer structure) is \emph{of class $4$} or $2$ accordingly.
In each case, we can give an explicit projective equation, as follows.

\begin{thm} \label{thm:projective equation of supersingular Kummer}
Let $\tilde{X}$ be as above.
Then there exist $4$ smooth rational curves $\Gamma_1, \dots, \Gamma_4$,
forming a $D_4$ configuration and disjoint from the $16$ curves of the Kummer structure,
such that the contraction $X'$ of the $20$ curves is of the following form.
\begin{itemize}
\item If $Q \cong Q_4$, then 
$X' = (w^2 = H(x, y))$ is an inseparable double covering of $\bP^1 \times \bP^1$, and $H$ is of the form
\begin{equation} \label{eq:H:class 4}
H = x^4 y + x y^4 + h_{30} x^3 + h_{21} x^2 y + h_{12} x y^2 + h_{03} y^3 + h_{11} xy,
\end{equation}
and $\Gamma_1, \dots, \Gamma_4$ are the exceptional curves over the RDP $P_{\infty}$ at $(x, y) = (\infty, \infty)$.
\item If $Q \cong Q_2$, then 
$X' = (y^2 = H(x, t))$ is a Weierstrass form of a quasi-elliptic surface, and $H$ is of the form
\begin{equation} \label{eq:H:class 2}
H = x^3 + h_{11} x t + h_{12} x t^2 + h_{03} t^3 + h_{05} t^5 + t^9,
\end{equation}
and $\Gamma_1, \dots, \Gamma_4$ are the exceptional curves over the RDP $P_{\infty}$ at $(x/t^4, 1/t) = (0, 0)$.
\end{itemize}
Conversely, let $X'$ be the surface defined by one of the two equations above. 
Then the exceptional curves other than the ones over $P_{\infty}$ 
form a Kummer structure on the minimal resolution $\tilde{X}$ of $X'$,
provided $\Sing(X')$ consists only of RDPs.
\end{thm}

Moreover, we can explicitly describe the configuration of the singularities and hence the type, in terms of the coefficients $h_{ij}$.
Define subspaces $W, Z_1, Z_2 \subset \bA^4$ by 
\begin{align*}
W &= \set{(u_0, u_1, u_2, u_3) \mid u_1 u_2 - u_0 u_3 = 0}, \\
Z_1 &= \set{(u_0, u_1, u_2, u_3) \mid u_1 u_2 - u_0 u_3 = u_1^2 - u_0 u_2 = u_2^2 - u_1 u_3 = 0}, \\
Z_2 &= \set{(u_0, u_1, u_2, u_3) \mid u_1 u_2 - u_0 u_3 = u_2^2 - u_0 u_1 = u_1^2 - u_2 u_3 = 0}.
\end{align*}
(We also note that $Z_1 \cap Z_2 = \set{(c a^3, c a^2 b, c a b^2, c b^3) \mid a, b \in \setF_4, c \in k}$.)
\begin{prop} \label{prop:singularities}
\begin{enumerate}
\item \label{item:singularities:class 4}
Let $X' = (w^2 = H(x, y))$ and $P_{\infty}$ as above, with 
$H$ as in equation \eqref{eq:H:class 4}. Then, letting $u = (h_{30}, h_{21}, h_{12}, h_{03})$, 
\[
\Sing(X') \setminus \set{P_{\infty}} = 
\begin{cases}
16 A_1   & (h_{11} \neq 0), \\
4 D_4^0  & (h_{11} = 0, u \notin W), \\
2 D_8^0  & (h_{11} = 0, u \in W, u \notin Z_1, u \notin Z_2), \\
1 D_{16}^0 & (h_{11} = 0, u \in W, u \notin Z_1, u    \in Z_2), \\
2 E_8^0  & (h_{11} = 0, u \in W, u    \in Z_1, u \notin Z_2), \\
1 \, \text{non-RDP} & (h_{11} = 0, u \in W, u \in Z_1, u \in Z_2).
\end{cases}
\]
\item \label{item:singularities:class 2} 
Let $X' = (y^2 = H(x, t))$ and $P_{\infty}$ as above, with 
$H$ as in equation \eqref{eq:H:class 2}. Then
\[
\Sing(X') \setminus \set{P_{\infty}} = 
\begin{cases}
16 A_1   & (h_{11} \neq 0), \\
4 D_4^0  & (h_{11} = 0, h_{03} \neq 0), \\
2 D_8^0  & (h_{11} = 0, h_{03} = 0, h_{12} \neq 0, h_{05} \neq 0), \\
1 D_{16}^0 & (h_{11} = 0, h_{03} = 0, h_{12} \neq 0, h_{05} = 0), \\
2 E_8^0  & (h_{11} = 0, h_{03} = 0, h_{12} = 0, h_{05} \neq 0), \\
1 \, \text{non-RDP} & (h_{11} = 0, h_{03} = 0, h_{12} = 0, h_{05} = 0).
\end{cases}
\]
\end{enumerate}
\end{prop}

\subsection{Spaces of $1$-forms}

\begin{prop} \label{prop:proportional 1-forms:ter}
Let $X$ be a contracted inseparable Kummer surface.
Then any two (nonzero) elements of $H^0(X^{\sm}, \Omega^1)$ are proportional,
that is, linearly dependent over $k(X)$.
\end{prop}
\begin{proof}
Since the $1$-form corresponding to $A \to X$ vanishes somewhere on $X^{\sm}$ (Proposition \ref{prop:covering of inseparable Kummer surface}),
the sheaf $\Omega^1_{X^{\sm}}$ is not free.

Let $\eta_1, \eta_2 \in H^0(X^{\sm}, \Omega^1)$ be two nonzero global sections.
If $\eta_1$ and $\eta_2$ are not proportional, then $\eta_1 \wedge \eta_2$ is a nonzero global section of $\Omega^2_{X^{\sm}} \cong \cO_{X^{\sm}}$, 
which implies $\Omega^1_{X^{\sm}}$ is free with basis $\eta_1, \eta_2$, contradiction.
\end{proof}

We also use the equality $\dim_k H^0(X^{\sm}, Z_{\infty} \Omega^1) = 5$ (Theorem \ref{thm:characterization of Kummer types using Z_infty}).

\subsection{Curves generating the orthogonal complement}
We keep the notation.
By Proposition \ref{prop:structure of Q}, the orthogonal complement $Q$ of $K$ in $\Pic(\tilde{X})$ is 
isomorphic to either $Q_2$ or $Q_4$.
We next show that the basis of $Q_4$ or $Q_2$ as in Definition \ref{def:lattice Q} is realized by smooth rational curves.
\begin{prop} \label{prop:curves generating Q}
There exist $6$ smooth rational curves $\Gamma_1, \dots, \Gamma_6$ generating $Q$
and having intersection numbers as in Definition \ref{def:lattice Q}.
\end{prop}
\begin{proof}
First suppose $Q \cong Q_4$.
Let $w_1, \dots, w_6 \in Q$ as in Definition \ref{def:lattice Q},
and let $D = 2 w_1 + \sum_{i = 2}^6 w_i $.
We have $D \cdot w_i = 0$ ($i = 2, \dots, 6$) and $D^2 = 2$.
Applying Proposition \ref{prop:nef up to Weyl} to $D \in Q$,
we may assume that $D$ is big and nef.
This induces an embedding of $K \oplus A_1^{\oplus 5}$ into $L := \spanned{D}^{\perp}$, which is of finite index (because both sides have rank $21$).
By Propositions \ref{prop:-2 divisors orthogonal to a big nef} and \ref{prop:embedding of Kummer lattices:ter},
there exist smooth rational curves $\Gamma_2, \dots, \Gamma_{m+1}$ such that
$\bigcup_{j=2}^{m+1} \Gamma_j$ is an ADE configuration disjoint from $C_1 \cup \dots \cup C_{16}$
and $A_1^{\oplus 5} \subset \spanned{[\Gamma_j]}$.
By considering the rank, we have $m = 5$.
Since $\spanned{w_2, \dots, w_6} \subset Q_4$ is saturated, the configuration of $\Gamma_2 \cup \dots \cup \Gamma_6$ is $5 A_1$.

It remains to show that $w_1$ is the class of a smooth rational curve.
Consider the embedding of $K \oplus \spanned{\Gamma_2, \Gamma_3, \Gamma_4, w_1}$.
This is contained in the span of the classes of $19 + l$ curves $C_1, \dots, C_{16}, \Gamma_2, \Gamma_3, \Gamma_4, F_1, \dots, F_l$, with $l \in \set{1, 2}$.
If $l = 1$ then we are done. Suppose $l = 2$.
The divisor $\Gamma_2 \cup \Gamma_3 \cup \Gamma_4 \cup F_1 \cup F_2$ is an ADE configuration.
It should be $D_5$ since it is connected and contains $D_4$ as a sublattice.
By considering positive roots of $D_5$, we conclude that 
$w_6 = [F_1] + [F_2]$, and exactly $2$ among $\Gamma_j$ ($j = 2, 3, 4$) intersects $F_1$. 
We derive a contradiction by considering all choice of $3$ curves among $\Gamma_j$ ($2 \leq j \leq 6$).

The case where $Q \cong Q_2$ can be proved similarly.
It is easier than the previous case because we see at once that $\Gamma_1, \Gamma_i, \Gamma_j, \Gamma_5, \Gamma_6$ form a $D_5$
for any choice of distinct $i, j$ in $\set{2, 3, 4}$.
\end{proof}

\subsection{Proof of Theorem \ref{thm:projective equation of supersingular Kummer} in the case of class 4} \label{subsec:proof of equation:class 4}

Suppose $Q \cong Q_4$.

Take $\Gamma_1, \dots, \Gamma_6$ as in Proposition \ref{prop:curves generating Q}.
The linear system defined by the divisor $4 \Gamma_1 + 2 (\Gamma_2 + \Gamma_3 + \Gamma_4) + \Gamma_5 + \Gamma_6$
(which is big and nef, and orthogonal to $\Gamma_1, \dots, \Gamma_4$)
induces a double covering 
$\map{\psi'}{X'}{\Sigma' := \bP^1 \times \bP^1}$,
and $X'$ is defined by an equation of the form $w^2 + I(x, y) w + H(x, y) = 0$
with 
\[
I \in \spanned{x^i y^j \mid 0 \leq i, j \leq 2}_k, \quad 
H \in \spanned{x^i y^j \mid 0 \leq i, j \leq 4}_k.
\]
Here the map $\psi'$ is $(w, x, y) \mapsto (x, y)$.
We note that $X'$ is the contraction of 
$C_1 \cup \dots \cup C_{16} \cup \Gamma_1 \cup \Gamma_2 \cup \Gamma_3 \cup \Gamma_4 \subset \tilde{X}$.
\begin{claim} \label{claim:I=0}
$I = 0$. In other words, the map $\psi'$ is inseparable.
\end{claim}
\begin{proof}
Let $V$ be the image of $H^0(X^{\sm}, Z_{\infty} \Omega^1) \subset H^0(X^{\sm}, \Omega^1) \to H^0(X'^{\sm}, \Omega^1)$.
Note that $\dim V = 5$ (Theorem \ref{thm:characterization of Kummer types using Z_infty}) and any two elements of $V$ are proportional (Proposition \ref{prop:proportional 1-forms:ter}).
Let $\Sigma'^{\circ} := \psi'(X'^{\sm}) \subset \Sigma'$.

The sheaf $\Omega^1_{\Sigma'}$ decomposes into the sum of rank $1$ sheaves $F_1, F_2$,
each consisting of proportional $1$-forms ($F_1$ to $dx$ and $F_2$ to $dy$).
By the proportionality of $V$,
the restriction maps $V \to \Hom(\Omega^1_{X'^{\sm}}, \cO_{X'^{\sm}}) \to \Hom(\psi^* F_i \restrictedto{\Sigma'^{\circ}}, \cO_{X'^{\sm}})$ 
to $\psi^* F_i \subset \psi^* \Omega^1_{\Sigma} \subset \Omega^1_{X'}$
($i = 1, 2$) are either injective or zero.
By the pullback-pushforward adjunction and
isomorphisms $F_1 \cong \cO(-2, 0)$ and $\psi_* \cO_{X'} \cong \cO \oplus \cO(-2, -2)$,
the image of $V$ lies in 
$\Hom(F_i, \psi_* \cO_{X'}) \cong H^0(\Sigma', \cO(2, 0) \oplus \cO(0, -2)) \cong H^0(\Sigma', \cO(2, 0))$.
Since $\dim H^0(\Sigma', \cO(2, 0)) < 5$, 
we conclude that this map is zero.
We obtain the same conclusion for $F_2$.
Therefore, any element of $V$ annihilates $\psi'^* (F_1 \oplus F_2) = \psi'^* \Omega^1_{\Sigma'^{\circ}} \subset \Omega^1_{X'^{\sm}}$.
If $I \neq 0$, equivalently, if $\psi'$ is separable, then this subsheaf is of full rank, hence $V = 0$, contradiction.
\end{proof}
Write $H = \sum_{0 \leq i, j \leq 4} h_{ij} x^{i} y^{j}$.
We may assume $h_{ij} = 0$ when $i$ and $j$ are both even.
Placing the image of $\Gamma_1 \cup \dots \cup \Gamma_4$, which is an RDP of type $D_4$, at $P_{\infty} = (x = \infty, y = \infty)$, 
we have $h_{33} = h_{34} = h_{43} = 0$.
Since $\Sing(X') \setminus \set{P_{\infty}}$ is disjoint with $\Gamma_5 \cup \Gamma_6 = (x = \infty) \cup (y = \infty)$,
we have $h_{41}, h_{14} \neq 0$, and we may assume $h_{41} = h_{14} = 1$.
By placing a singularity at $(0, 0)$, we obtain $h_{10} = h_{01} = 0$. 

We want to show that $h_{31}, h_{32}, h_{13}, h_{23}$ are all zero,
using the assumption that the exceptional curves form a Kummer structure.
To show this, we compute the action of the Cartier operator $C$ explicitly.
We identify $H^0(X^{\sm}, \Omega^1)$ and its subspaces, such as $H^0(Z_i) := H^0(X^{\sm}, Z_i \Omega^1)$,
with their images in $H^0(X'^{\sm}, \Omega^1)$.
Let $\eta_0 = dx/H_y = dy/H_x \in H^0(X'^{\sm}, \Omega^1)$.
We have bases
\begin{align*}
H^0(X^{\sm}, \Omega^1) = H^0(Z_0) &= \spanned{x^i y^j \mid (i, j) \in I_1}_k \cdot \eta_0 \oplus kw \cdot \eta_0 \\
\supset H^0(Z_1) &= \spanned{x^i y^j \mid (i, j) \in I_1}_k \cdot \eta_0 \\
\supset H^0(Z_2) &= \spanned{x^i y^j \mid (i, j) \in I_2}_k \cdot \eta_0,
\end{align*}
where 
\begin{align*}
I_1 &:= \set{(0, 0), (1, 0), (2, 0), (0, 1), (0, 2), (1, 1)}, \\
I_2 &:= \set{(0, 0), (1, 0), (2, 0), (0, 1), (0, 2)}.
\end{align*}
Here we used Lemma \ref{lem:explicit formula for Cartier operator:general} to determine $H^0(Z_2) = C^{-1}(H^0(Z_1))$.
By the same lemma, the explicit formula for $\map{C}{H^0(Z_2)}{H^0(Z_1)}$ is 
\begin{gather*}
C \biggl( \bigl( \sum_{(i, j) \in I_2} g_{ij} x^i y^j \bigr) \eta_0 \biggr)
= \bigl( \sum_{(i, j) \in I_1} \sqrt{f_{ij}} x^i y^j \bigr) \eta_0, \\
f_{ij} := \sum_{i_1 + i_2 = 2 i + 1, \; j_1 + j_2 = 2 j + 1} h_{i_1 j_1} g_{i_2 j_2}.
\end{gather*}
We observe that $H^0(Z_3) = H^0(Z_2)$ 
if and only if all coefficients of $f_{11}$ are zero, 
that is, $(h_{31}, h_{32}, h_{13}, h_{23}) = (0, 0, 0, 0)$.
But since (by Theorem \ref{thm:characterization of Kummer types using Z_infty}) 
$5 = \dim H^0(Z_{\infty}) \leq \dim H^0(Z_3) \leq \dim H^0(Z_2) = 5$, we indeed have $H^0(Z_3) = H^0(Z_2)$.

\bigskip

We now show the converse statement.
Consider the surface $X' = (w^2 = H(x, y))$ with $H$ as in equation \eqref{eq:H:class 4},
and $X$ be the minimal resolution of $P_{\infty}$.
First we confirm (the case of class $4$ of) Proposition \ref{prop:singularities}.

It is clear that no point on $(x = \infty) \cup (y = \infty)$ other than $P_{\infty}$ is a singular point.
It is straightforward to check that the point $(x, y) = (0, 0)$ is of the type as stated.
We claim that $\Aut(X')$ acts on $\Sing(X') \setminus \set{P_{\infty}}$ transitively,
from which it follows that every singularity is of the same type.
Take a singular point $P$ at $(x, y, w) = (\alpha, \beta, \gamma)$.
We can write
\[
H(x, y) = H(x - \alpha, y - \beta) + \sum c_{ij} (x - \alpha)^i (y - \beta)^j
\]
for some $c_{ij} \in k$,
where the range of the index $(i, j)$ is \[ \set{(0, 0), (1, 0), (2, 0), (4, 0), (0, 1), (0, 2), (0, 4)}. \]
But since $P$ is a singular point,
we have $c_{10} = c_{01} = 0$, and all remaining indices are pairs of even integers.
Hence $H(x, y) = H(x - \alpha, y - \beta) + f(x, y)^2$ for a suitable polynomial $f$,
and thus $(x, y, w) \mapsto (x - \alpha, y - \beta, w - f(x, y))$ is an automorphism taking $P$ to the origin.

By computing the intersection number of $(H_x = 0)$ and $(H_y = 0)$ on $\bP^2$
and comparing it with the Tyurina numbers of the singularities,
we see that the sum of the indices of the RDPs including $P_{\infty}$ is $20$, provided there is no non-RDP.
We conclude that $\Sing(X') \setminus \set{P_{\infty}}$ is as stated.
When a non-RDP singularity exists, its Tyurina number is $32$ and the argument is parallel.

Since $H^0(Z_{\infty}) = H^0(Z_2)$ we have $\dim H^0(Z_{\infty}) = 5$,
and then by Theorem \ref{thm:characterization of Kummer types using Z_infty}
the $16$ exceptional curves form a Kummer structure.

\subsection{Proof of Theorem \ref{thm:projective equation of supersingular Kummer} in the case of class 2} \label{subsec:proof of equation:class 2}

We now consider the case where $Q \cong Q_2$.
The proof is parallel to the previous case ($Q \cong Q_4$).

Take $\Gamma_1, \dots, \Gamma_6$ as in Proposition \ref{prop:curves generating Q}.
The linear system defined by the divisor $2 \Gamma_1 + (\Gamma_2 + \dots + \Gamma_5)$
(which is orthogonal to $\Gamma_1, \dots, \Gamma_5$)
induces a genus $1$ fibration having this divisor as a fiber and $\Gamma_6$ as a section. Let
\[ y^2 + I(x, t) y + H(x, t) = 0 \]
be its Weierstrass equation with 
\[
I \in \spanned{x t^0, \dots, x t^2, t^0, \dots, t^6}_k, \quad
H - x^3 \in \spanned{x t^0, \dots, x t^8, t^0, \dots, t^{12}}_k.
\]
Let $X'$ be the surface defined by this equation, which is the contraction of 
$C_1 \cup \dots \cup C_{16} \cup \Gamma_1 \cup \Gamma_2 \cup \Gamma_3 \cup \Gamma_4 \subset \tilde{X}$.
Write $\map{\psi'}{X'}{\Sigma' := \Sigma_4}$ the map $(y, x, t) \mapsto (x, t)$.
Here $\Sigma_n$ denotes the Hirzebruch surface.

As in Claim \ref{claim:I=0}, we want to show $I = 0$.
Let $\Sigma' \to \bP^1$ be the $\bP^1$-fibration, 
$f$ the image of $\Gamma_5$ (which is a fiber of the fibration),
and $C_0$ the section with $C_0^2 = -4$ (which is equal to $(x = \infty)$).
Let $t$ be a coordinate on $\bP^1$ such that $f = (t = \infty)$.

Let $V \subset H^0(X'^{\sm}, \Omega^1)$ as in the proof of Claim \ref{claim:I=0}.
Since $\Omega^1_{\Sigma'} \cong \cO_{\Sigma'}(-2f) \oplus \cO_{\Sigma'}(-2 C_0 - 4 f)$,
we have to consider the images of $V$ in 
$H^0(\Sigma', \cO_{\Sigma'}(2f)) = \spanned{1, t, t^2}$ and 
$H^0(\Sigma', \cO_{\Sigma'}(2 C_0 + 4 f)) = \spanned{1, t, t^2, t^3, t^4, x}$.
Since $V$ comes from $H^0(X^{\sm}, \Omega_X^1)$, the image should vanish twice at the image of $\Gamma_1 \cup \dots \cup \Gamma_4$, 
which is the point $(1/t = 0, x/t^4 = 0)$. 
Then it follows that the image of $V \to H^0(\Sigma', \cO_{\Sigma'}(2 C_0 + 4 f))$ is actually contained in $\spanned{1, t, t^2, x}$,
and we can conclude as in the case of $Q \cong Q_4$ that $I = 0$.

We may assume 
\[
H = x^3 + (h_{11} t + h_{12} t^2 + h_{13} t^3 + h_{15} t^5 + h_{16} t^6) x + (h_{03} t^3 + h_{05} t^5 + h_{07} t^7 + t^9).
\]
We want to show $h_{13} = h_{15} = h_{16} = h_{07} = 0$.

As in the case of class $4$, we describe the Cartier operator.
Let $\eta_0 = dx/H_t = dt/H_x \in H^0(X'^{\sm}, \Omega^1)$.
We have bases
\begin{align*}
H^0(X^{\sm}, \Omega^1) = H^0(Z_0) &= \spanned{x^i t^j \mid (i, j) \in I_1}_k \cdot \eta_0 \oplus k y \cdot \eta_0 \\
\supset H^0(Z_1) &= \spanned{x^i t^j \mid (i, j) \in I_1}_k \cdot \eta_0 \\
\supset H^0(Z_2) &= \spanned{x^i t^j \mid (i, j) \in I_2}_k \cdot \eta_0,
\end{align*}
where 
\begin{align*}
I_1 &:= \set{(1, 0), (0, 0), (0, 1), (0, 2), (0, 3), (0, 4)}, \\
I_2 &:= \set{(1, 0), (0, 0), (0, 1), (0, 2), (0, 4)},
\end{align*}
and the explicit formula 
\begin{gather*}
C \biggl( \bigl( \sum_{(i, j) \in I_2} g_{ij} x^i t^j \bigr) \eta_0 \biggr)
= \bigl( \sum_{(i, j) \in I_1} \sqrt{f_{ij}} x^i t^j \bigr) \eta_0, \\
f_{ij} := \sum_{i_1 + i_2 = 2 i + 1, \; j_1 + j_2 = 2 j + 1} h_{i_1 j_1} g_{i_2 j_2}.
\end{gather*}
We observe that $H^0(Z_3) = H^0(Z_2)$ 
if and only if all coefficients of $f_{03}$ are zero, 
that is, $(h_{13}, h_{15}, h_{16}, h_{07}) = (0, 0, 0, 0)$.

The rest of the proof is parallel.

\section{Structure of the abelian-like covering} \label{sec:structure of the covering}

Let $X$ be a contracted inseparable Kummer surface. 
The covering(s) $A \to X$ constructed in Proposition \ref{prop:covering of inseparable Kummer surface} is never normal.
Using the projective equation of Theorem \ref{thm:projective equation of supersingular Kummer},
we show (Theorem \ref{thm:group structure of abelian-like covering}) that the smooth locus $A^{\sm}$ of $A$ admits a structure of abelian group variety,
which would justify the name of \emph{abelian-like} coverings.
When the covering $A$ is not unique, we have to choose a suitable one. 

In Section \ref{subsec:normalization}, we describe the relation between $\cO_X$, $\cO_A$, and $\cO_{\normalization{A}}$, 
where $\normalization{\functorspace}$ denotes the normalization.
In Section \ref{subsec:Kummer quotients}, we discuss, for a fixed $\cO_{\normalization{A}}$, 
which $G$-actions ($G = \mu_2, \alpha_2$) on which subalgebras $\cO_A \subset \cO_{\normalization{A}}$ induce inseparable Kummer quotients.

\subsection{Group structure} \label{subsec:group structure}

\begin{thm} \label{thm:group structure of abelian-like covering}
Let $X$ be a contracted inseparable Kummer surface.
If the Kummer structure is of type $16 A_1$ (resp.\ of additive type), 
then the smooth locus $A^{\sm}$ of $A$,
where $A$ is the $\mu_2$-covering (resp.\ one of the $\alpha_2$-coverings) of $X$,
admits a structure of a group variety isomorphic to $\Ga^{\oplus 2}$,
and $\Fix(G)$ is a finite subgroup scheme.
\end{thm}

\begin{proof}
We use the equations of Theorem \ref{thm:projective equation of supersingular Kummer}.
First consider the case of class $4$.

We have $\Sing(A) = \Delta_1 \cup \dots \cup \Delta_6$,
where $\Delta_5$ and $\Delta_6$ are the inverse images of the strict transforms of the curves $(x = \infty), (y = \infty) \subset X'$ and 
$\Delta_1 \cup \dots \cup \Delta_4$ is the inverse image of the exceptional divisor of $X \to X'$.

First suppose $h_{11} \neq 0$, equivalently, the Kummer structure is of type $16 A_1$.
We have $A^{\sm} = \Spec k[\sqrt{x}, \sqrt{y}]$.
The derivation $D$ corresponding to the $\mu_2$-action on $A$ is given by 
$D(\sqrt{x}) =  \sqrt{h_{11}^{-1} H_y}, D(\sqrt{y}) = \sqrt{h_{11}^{-1} H_x}$,
and hence the fixed locus $\Fix(G)$ is the closed subscheme $(\sqrt{H_x} = \sqrt{H_y} = 0)$.
Since all the monomials in $H_x$ and $H_y$ are of the form $x^{2^e}$ or $y^{2^e}$, 
the fixed locus is a finite subgroup scheme of $\Ga^{\oplus 2}$
under the identification $\Spec k[\sqrt{x}, \sqrt{y}] \cong \Ga^{\oplus 2}$ given in the obvious way.
In this case, the scheme $\Fix(G)$ is reduced.

Now suppose $h_{11} = 0$, equivalently, the Kummer structure is of additive type.
The set of $1$-forms $\eta$ corresponding to $\alpha_2$-coverings 
is the complement of the union of finitely many hyperplanes in 
$\set{L^2 \eta_0 \mid L \in \spanned{1, x, y}_k} = H^0(X^{\sm}, B_1 \Omega^1) \subset H^0(X^{\sm}, \Omega^1)$.
If $Y$ is the covering corresponding to $L^2 \eta_0$, then 
$\cO_Y$ is generated by $\frac{L}{\sqrt{H_y}} \sqrt{x}$ and $\frac{L}{\sqrt{H_x}} \sqrt{y}$ over $\Spec k[w,x,y]/(w^2 + H) \subset X$.

Among $\alpha_2$-coverings of $X$, we denote by $A$ the one corresponding to the $1$-form $\eta_0$.
Then it is straightforward to observe that $\Sing(A)$ is exactly the points above $(x = \infty) \cup (y = \infty)$ 
and that $A^{\sm} = \Spec k[\sqrt{x}, \sqrt{y}]$.
The derivation $D$ corresponding to $\alpha_2$-action on $A$ is given by 
$D(\sqrt{x}) = \sqrt{H_y}, D(\sqrt{y}) = \sqrt{H_x}$.
As in the case of $16 A_1$, we conclude that $\Fix(G)$ is a finite subgroup scheme of $A^{\sm} \cong \Ga^{\oplus 2}$,
non-reduced in this case.

The case of class $2$ is done in a parallel way.
\end{proof}

\begin{defn}
The \emph{abelian-like covering} of $X$ is the covering $A \to X$ in the proof of Theorem \ref{thm:group structure of abelian-like covering}.
\end{defn}

\begin{rem} \label{rem:group structure using 1-forms}
Here is another way to describe the group structure (we omit the proof).
Choose $P \in \Sing(X)$.
Take $\eta_0$ as in Sections \ref{subsec:proof of equation:class 4} or \ref{subsec:proof of equation:class 2}.
Let $\eta_1, \dots, \eta_4$ be a basis of the space 
$\set{\eta \in H^0(X^{\sm}, Z_{\infty} \Omega^1) \mid \text{$\eta$ vanishes at $P$}}$.
Since any two elements of $H^0(X^{\sm}, \Omega^1)$ are proportional (Proposition \ref{prop:proportional 1-forms:ter}), 
the elements $\eta_0, \dots, \eta_4$ define a rational map $\map{\Phi}{X}{\bP^4}$, which is a morphism.
Let $U = \set{(W_0 : \dots : W_4) \mid W_0 \neq 0} \subset \bP^4$.
Then $\Phi^{-1}(U) = X \setminus \Supp(\Gamma)$, and 
$\map{\Phi \restrictedto{\Phi^{-1}(U)}}{\Phi^{-1}(U)}{U \cap \Phi(X)}$ coincides with $X \setminus \Supp(\Gamma) \to \thpower{(A^{\sm})}{2}$.
Here, if $X$ is of additive type, $A$ is the one chosen in the proof of Theorem \ref{thm:group structure of abelian-like covering}.
We claim that there exist $2$ elements 
$f_1, f_2$ of the form $f_i = \sum_{j = 1}^4 (c_{ij} w_j + d_{ij} w_j^2)$ with $c_{ij}, d_{ij} \in k$, where $w_j = W_j/W_0$,
that generate the ideal defining $U \cap \Phi(X) \subset U = \Spec k[w_1, \dots, w_4]$.
Indeed, this assertion does not depend on the choice of the basis $\eta_1, \dots, \eta_4$,
and it clearly holds if we take 
$x \eta_0, x^2 \eta_0, y \eta_0, y^2 \eta_0$ 
(resp.\ $x \eta_0, t \eta_0, t^2 \eta_0, t^4 \eta_0$)
if $X$ is of class $4$ (resp.\ $2$).
It follows that $U \cap \Phi(X)$ is a subgroup scheme of $U \cong \Ga^{\oplus 4}$.
\end{rem}

\subsection{Normalization} \label{subsec:normalization}

Let $\map{\pi}{A}{X}$ be the abelian-like covering of a contracted inseparable Kummer surface, 
and $\Delta$ be the pullback of $\Zero(\eta_0)$ to the normalization $\normalization{A}$ of $A$.
Then we have $\cO_A = \cO_X + \cO_{\normalization{A}}(-\Delta)$: 
this equality is obvious outside $\Supp(\Delta)$,
and follows from Proposition \ref{prop:normalization} over $X^{\sm}$.

\begin{prop}
Let $X, A, \normalization{A}$, and $\Delta$ be as above,
and suppose $X$ is of class $b \in \set{4, 2}$.
Then the surface $\normalization{A}$ and the divisor $\Delta$ on $\normalization{A}$ 
depends only on $b$, not on the contracted inseparable Kummer surface $X$.
\end{prop}

Let $\map{\psi'}{X'}{\Sigma'}$ be the covering given in Sections \ref{subsec:proof of equation:class 4} and \ref{subsec:proof of equation:class 2},
and $\map{\psi}{X}{\Sigma}$ be the map induced by $\psi'$ 
(that is, $\psi$ is a homeomorphism and $\cO_{\Sigma} = \cO_X \cap k(\Sigma')$).
We observe that $\thpower{k(A)}{2} = k(\Sigma)$ and hence $\thpower{\normalization{A}}{2} = \Sigma$.
Thus it suffices to prove the assertion for $\Sigma$ and the corresponding divisor, which we also write $\Delta$.
We only give the descriptions and omit the verifications.

Suppose $b = 4$.
Let $\Sigma''$ be the blowup of $\Sigma' = \bP^2 \supset \Spec k[x, y]$ at the point $(x', y') = (0, 0)$, where $x' = 1/x$ and $y' = 1/y$,
and $\Sigma$ the blowup of $\Sigma''$ at the points $(x'/y')^3 = 1$ on the exceptional curve.
Let $\Delta_5$ and $\Delta_6$ be the strict transforms of $(x' = 0), (y' = 0) \subset \Sigma'$,
$\Delta_1$ the strict transform of the exceptional curve of $\Sigma'' \to \Sigma'$,
and $\Delta_2, \Delta_3, \Delta_4$ the exceptional curves of $\Sigma \to \Sigma''$.
Then $\Delta = 3 \Delta_1 + 2 (\Delta_2 + \dots + \Delta_6)$.

Suppose $b = 2$.
Let $\Sigma''$ be the blowup of $\Sigma' = \Sigma_4 \supset \Spec k[x, t]$ at the point $(x', t') = (0, 0)$, where $x' = x/t^4$ and $t' = 1/t$,
and $\Sigma$ the blowup of $\Sigma''$ at the points $(x'/t')^3 = 1$ on the exceptional curve.
Let $\Delta_5$ and $\Delta_6$ be the strict transforms of $(t' = 0), (1/x = 0) \subset \Sigma'$,
$\Delta_1$ the strict transform of the exceptional curve of $\Sigma'' \to \Sigma'$,
and $\Delta_2, \Delta_3, \Delta_4$ the exceptional curves of $\Sigma \to \Sigma''$.
Then $\Delta = 5 \Delta_1 + 4 (\Delta_2 + \dots + \Delta_4) + 6 \Delta_5 + 2 \Delta_6$.

\subsection{Kummer quotients of the abelian-like covering} \label{subsec:Kummer quotients}

If an involution of an abelian surface has nonempty $0$-dimensional fixed locus,
then it is the inversion with respect to some choice of the origin,
and the minimal resolution of the quotient is a Kummer surface (K3 or not).
We now consider an inseparable analogue of this fact.
A substantial difference is that, in general, our $A$ admits many non-isomorphic inseparable Kummer quotients.

\begin{lem} \label{lem:derivation regular on Delta}
Suppose $D$ is a rational derivation on a variety $Y$ in characteristic $2$
satisfying $D^2 = c D$ for some $c \in k$.
Suppose $\cI \subset \cO_Y$ is a sheaf of ideals satisfying $D(\cI) \subset \cO_Y$.
Let $\cR = \Ker(\map{D \restrictedto{\cO_Y}}{\cO_Y}{k(Y)})$ and $\cS = \cR + \cI$ 
(the sum is taken as subsheaves of $k$-vector spaces of $\cO_Y$).
Then $\cR \subset \cS \subset \cO_Y$ are subsheaves of $k$-algebras,
$D$ is a regular derivation on $\cS$, and $\cS^D = \cR$.
The morphism $\sSpec \cS \to \sSpec \cR$ is a $G$-quotient morphism,
with $G = \mu_2$ if $c \neq 0$ and $G = \alpha_2$ if $c = 0$.
\end{lem}
\begin{proof}
$\cS$ is closed under $D$ since $D(\cI) \subset (D - c)(\cI) + c \cI \subset \cR + \cI$.
The remaining assertions are straightforward.
\end{proof}

Now fix $b \in \set{4, 2}$.
Let $\normalization{A}$ and $\Delta$ be as in Section \ref{subsec:normalization}, defined from an inseparable Kummer surface of class $b$.
Despite the notation, they depend only on $b$ and not on $X$ nor $A$.
We regard $A^{\sm} := \normalization{A} \setminus \Supp(\Delta)$ as a $\Ga^{\oplus 2}$-torsor by using $x, y$ (resp.\ $x, t$) as coordinates if $b = 4$ (resp.\ $b = 2$).
For each closed point $P$ of $A^{\sm}$, 
we can equip $A^{\sm}$ with a structure of a group variety by declaring $P$ to be the origin.

We consider the following conditions (\ref{item:k-p-closed})--(\ref{item:fixed locus is a subgroup})
on a rational derivation $D$ on $\normalization{A}$.
\begin{enumerate}[label=(\roman*), ref=\roman*]
\item \label{item:k-p-closed}
$D^2 = c D$ for some $c \in k$.
\item \label{item:regular on -Delta}
$D(\cO_{\normalization{A}}(-\Delta)) \subset \cO_{\normalization{A}}$.
\end{enumerate}
Assuming (\ref{item:k-p-closed}) and (\ref{item:regular on -Delta}), 
we can apply Lemma \ref{lem:derivation regular on Delta} to $D$ and $\cO_{\normalization{A}}(-\Delta)$,
and we write $A = \sSpec \cS$, $X = \sSpec \cR = A^D$.
The surfaces $\normalization{A}$ and $A^{\sm}$ defined above coincide with the normalization and the smooth part of $A$ respectively.
\begin{enumerate}[label=(\roman*), ref=\roman*]
\addtocounter{enumi}{2}
\item \label{item:0-dimensional fixed locus} 
$\Fix(D \restrictedto A)$ is at most $0$-dimensional and contained in $A^{\sm}$. 
\item \label{item:fixed locus is a subgroup}
Moreover, $\Fix(D \restrictedto A)$ is a subgroup scheme of $A^{\sm} \cong \Ga^{\oplus 2}$ for a suitable choice of origin.
\end{enumerate}

\begin{prop} \label{prop:derivations on normalization}
Let $b$ and $\normalization{A}$ be as above.
Suppose $D$ satisfies (\ref{item:k-p-closed})--(\ref{item:0-dimensional fixed locus}). 
\begin{enumerate}
\item \label{item:all derivations}
If $b = 4$ (resp. $b = 2$),
then there exists $H \in k[x, y]$ as in equation \eqref{eq:H:class 4:general}
(resp.\ $H \in k[x, t]$ as in equation \eqref{eq:H:class 2:general})
such that $D$ is a nonzero scalar multiple of 
$H_y \partialdd{x} + H_x \partialdd{y}$ 
(resp.\ $H_t \partialdd{x} + H_x \partialdd{t}$).
\item \label{item:derivations whose fixed locus is a subgroup}
If $b = 4$, then (\ref{item:fixed locus is a subgroup}) always holds.
If $b = 2$, then (\ref{item:fixed locus is a subgroup}) holds if and only if $h_{07} = 0$.
\end{enumerate}
\end{prop}
Case $b = 4$:
\begin{gather} \label{eq:H:class 4:general}
H = x^4 y + x y^4 + \sum_{(i, j) \in I} h_{ij} x^i y^j,
\\ I = \set{(3, 0), (2, 1), (1, 2), (0, 3), (1, 1), \underline{(1, 0)}, \underline{(0, 1)}}. \notag
\end{gather}

Case $b = 2$:
\begin{gather} \label{eq:H:class 2:general}
H = x^3 + t^9 + \sum_{(i, j) \in I} h_{ij} x^i t^j, 
\\ I = \set{\underline{(2, 1)}, (1, 1), (1, 2), \underline{(1, 4)}, (0, 3), (0, 5), (0, 7), \underline{(1, 0)}, \underline{(0, 1)}}. \notag
\end{gather}

The coefficients for the underlined elements in $I$ can be eliminated by twisting by automorphisms, see Remark \ref{rem:simplification}.

\begin{proof}
(\ref{item:all derivations}) This is proved by a tedious computation.

Sketch in the case $b = 4$:
We have $D = f \partialdd{x} + g \partialdd{y}$ with $f, g \in k[x, y]$,
and $f$ and $g$ have no common divisor by assumption (\ref{item:0-dimensional fixed locus}).
This and the equalities $c f = f_x f + f_y g$, $c g = g_x f + g_y g$ (which follow from assumption (\ref{item:k-p-closed}))
imply that $f_x = c$, $f_y = 0$, $g_x = 0$, $g_y = c$.
This in particular implies that there exists $H$ with $H_x = g$, $H_y = f$, and $H_{xy} = c$,
and we may assume $H$ has no monomials $x^i y^j$ with both $i$ and $j$ even.
Assumption (\ref{item:regular on -Delta}) at the divisors $\Delta_5$ and $\Delta_6$ imply that
$f$ and $g$ are of the form
$f = \sum_{0 \leq i \leq 4, 0 \leq j \leq 2} f_{ij} x^i y^j$,
$g = \sum_{0 \leq i \leq 2, 0 \leq j \leq 4} g_{ij} x^i y^j$.
By assumption (\ref{item:regular on -Delta}) at the divisors $\Delta_2, \Delta_3, \Delta_4$,
these conditions should hold after coordinate changes of the form $y' = y + a x$ for any $a \in \setF_4$.
It follows that $f_{42} = f_{22} = g_{24} = g_{22} = 0$ and $f_{40} = g_{04}$.
Since it has no fixed points on $\Supp(\Delta)$, we obtain $f_{40} \neq 0$.

Case $b = 2$ is done similarly.

(\ref{item:derivations whose fixed locus is a subgroup})
By Theorem \ref{thm:group structure of abelian-like covering}, we have the group structure if $b = 4$ or if $b = 2$ and $h_{07} = 0$.
Conversely, suppose $b = 2$ and $\Fix(D)$ is a subgroup of $\Spec k[x, t] \cong \Ga^{\oplus 2}$ for a suitable choice of origin.
The $t$-coordinates of the points of $\Fix(D)$, counted with multiplicity, are exactly the roots of 
\[
f(t) := h_{11}^2 H_x + H_t^2 = h_{11}^3 t + h_{11}^2 h_{12} t^2 + \sum_{j = 3,5,7,9} h_{0j}^2 t^{2(j-1)}.
\]
Then $f(t) = h_{09}^2 g(t)^m$, where $m$ is the multiplicity of each point of $\Fix(D)$ and 
$g(t)$ is the monic polynomial whose roots are the elements of the subgroup $\Fix(D)_{\red} \subset \Ga$.
It is easy to show that then nonzero coefficients of $g$ appear only on degree $0$ or $2^e$.
Since $m = \deg f / \deg g$ is a power of $2$, the same holds for $f$.
Hence $h_{07} = 0$.
\end{proof}

\begin{rem} \label{rem:simplification}
In this remark, an \emph{automorphism} is that of the surface $\normalization{A}$ preserving every component of $\Delta$.

If $b = 4$ then, 
up to an automorphism of the form $(x, y) \mapsto (x - a, y - a')$,
we can assume that $h_{10} = h_{01} = 0$, hence it is as in equation \eqref{eq:H:class 4},
and $\Sing(X)$ is as in Proposition \ref{prop:singularities}(\ref{item:singularities:class 4}).

If $b = 2$ then, 
up to an automorphism of the form $(x, t) \mapsto (x - a_0 - a_1 t - a_2 t^2, t - a')$,
we can assume that $h_{21} = h_{14} = h_{10} = h_{01} = 0$.
Hence, it is as in equation \eqref{eq:H:class 2} if and only if $h_{07} = 0$,
and in that case $\Sing(X)$ is as in Proposition \ref{prop:singularities}(\ref{item:singularities:class 2}).

Note that these automorphisms preserve the group structure up to changing the origin.
\end{rem}

\begin{prop}
Let $b$ and $\normalization{A}$ be as above,
and suppose $D$ is a derivation satisfying conditions (\ref{item:k-p-closed})--(\ref{item:fixed locus is a subgroup}).
If $\Sing(X)$ consists only of RDPs, then $X$ is a contracted inseparable Kummer surface with abelian-like covering $A$.
\end{prop}
\begin{proof}
By Proposition \ref{prop:derivations on normalization} and Remark \ref{rem:simplification}, 
$X$ admits an equation as in Theorem \ref{thm:projective equation of supersingular Kummer}.
\end{proof}

\begin{prop} \label{prop:dependence on h of A}
Let $b$ and $\normalization{A}$ be as above,
and suppose $D$ is a derivation satisfying conditions (\ref{item:k-p-closed})--(\ref{item:0-dimensional fixed locus})
(not necessarily (\ref{item:fixed locus is a subgroup})),
with the simplification in Remark \ref{rem:simplification} applied.

If $b = 4$, then the subalgebra $\cO_A \subset \cO_{\normalization{A}}$ depends only on $h_{30}, h_{21}, h_{12}, h_{03}$, and not on $h_{11}$.

If $b = 2$, then the subalgebra $\cO_A \subset \cO_{\normalization{A}}$ depends only on $h_{12}, h_{05}, h_{07}$, and not on $h_{11}, h_{03}$. 
\end{prop}
\begin{proof}
Sketch:
Let $U = \Spec R \subset \normalization{A}$ be an affine open subscheme. 
Let $t \in R$ be a generator of $\cO_{\normalization{A}}(-\Delta)$ and
$v$ a generator of the $\thpower{R}{2}$-module $k(\thpower{A}{2}) H \cap R$.
Then we have $\Gamma(U, \cO_A) = (\thpower{R}{2} \oplus \thpower{R}{2} v) + t R$.

We compute such a presentation at a general point of each component of $\Delta$,
and check which coefficients of $H$ can be recovered from $\thpower{R}{2} v$ modulo $t R$.
It turns out that some coefficients cannot be recovered.
\end{proof}
\begin{rem}
To be precise, from $\cO_A$ we can recover only the ratio of the mentioned coefficients
in the sense that $(h_{30} : h_{21} : h_{12} : h_{03}) = (r h_{30} : r h_{21} : r h_{12} : r h_{03})$ for any $r \in k^*$
(resp.\ $(h_{12} : h_{05} : h_{07}) = (r^2 h_{12} : r^2 h_{05} : r h_{07})$).
\end{rem}

\section{Comparison with other works} \label{sec:final remarks}

\subsection{Kondo--Schr\"oer's construction} \label{subsec:comparison with Kondo--Schroer}
Our family of class $4$ (given in Theorem \ref{thm:projective equation of supersingular Kummer}) can be viewed as a generalization of the constructions of 
Schr\"oer \cite{Schroer:Kummer 2} and Kondo--Schr\"oer \cite{Kondo--Schroer:Kummer 2}, which we now recall.

Let $C_i = \Spec k[u_i^2, u_i^3] \cup \Spec k[u_i^{-1}]$ ($i = 1, 2$), which are cuspidal rational curves, and we denote the cusp by $0 \in C_i$.
Let $A' = C_1 \times C_2$, which is a non-normal surface with $\Sing(A') = (\set{0} \times C_2) \cup (C_1 \times \set{0})$.
Schr\"oer and Kondo--Schr\"oer gave certain families of vector fields $D$ (of additive or multiplicative type) on $A'$
and considered the quotients $X' = A'^D$.
We consider only the case where $X'$ is an RDP K3 surface.
In Schr\"oer's case, $X'$ is given by $X' = (w^2 = H(x,y))$ (where $x = u_1^{-2}$, $y = u_2^{-2}$) using elements $H$ of the form
\[ H = x^4 y + x y^4 + h_{21} x^2 y + h_{12} x y^2, \] 
$(h_{21}, h_{12}) \neq (0, 0)$,
and $\Sing(X')$ is $5 D_4^0$ or $D_4^0 + 2 D_8^0$.
In Kondo--Schr\"oer's case, $H$ is of the form 
\[ H = x^4 y + x y^4 + h_{21} x^2 y + h_{12} x y^2 + h_{11} x y + h_{10} x + h_{01} y,  \] 
and $\Sing(X')$ is as in the previous case if $h_{11} = 0$ 
and $D_4^0 + 16 A_1$ if $h_{11} \neq 0$.
Note that, by a suitable coordinate change of the form $u_i^{-1} \mapsto u_i^{-1} + a_i$, we may assume $h_{10} = h_{01} = 0$ in Kondo--Schr\"oer's case.
Hence their families are special cases of our inseparable Kummer surfaces of class $4$ ($h_{03} = h_{30} = 0$).

A conceptual difference is that they considered $G$-coverings $A' = C_1 \times C_2 \to X'$
but we view $A \to X$ as fundamental objects, 
where $X \to X'$ is the minimal resolution of the RDP of type $D_4$ at $P_{\infty} = (x = y = \infty)$.
While their $A'$ is constant on their family,
our $A$ also depends on the coefficients $h_{21}$ and $h_{12}$, as seen in Proposition \ref{prop:dependence on h of A}.

They also showed that every supersingular K3 surface in characteristic $2$ of Artin invariant at most $3$ 
are realized as (the minimal resolution of) their construction.
This can be compared with Theorem \ref{thm:characterization of supersingular Kummer}.

\subsection{Good reduction of usual Kummer surfaces} \label{subsec:supersingular reduction}

Let $\cO_K$ be a Henselian discrete valuation ring of mixed characteristic $(0, p)$,
and $K = \Frac \cO_K$.
Suppose that $A_K$ is an abelian surface over $K = \Frac \cO_K$,
that it has good reduction over $\cO_K$, which means that the reduction $A_k$ is an abelian surface over $k$.
Then, after replacing $\cO_K$ with a finite extension, 
the Kummer surface $\Km(A_K)$ has good reduction, i.e.,
there is a proper smooth algebraic space $\cX$ over $\cO_K$ with $\cX \otimes_{\cO_K} K \cong \Km(A_K)$.
This was known if either $p \neq 2$ or $A_k$ is non-supersingular, 
in which case the special fiber $X_k$ is naturally isomorphic to $\Km(A_k)$.
The case where $p = 2$ and $A_k$ is supersingular is our recent result (\cite{Matsumoto:kummerred}*{Theorem 1.2}).
Moreover, our proof shows that in that case $\cX_k$ (which is automatically a K3 surface) is supersingular and
that the image by the specialization map of $\Sing(A_K/{\pm 1})$
gives one of the following configuration of curves on the special fiber $\cX_k$:
$16 A_1, 4 D_4, 2 D_8, D_{16}, 2 E_8$,
and that $\cX_k$ is naturally equipped with a Kummer structure of corresponding type.
In fact, the equations of $\cX_k$ (\cite{Matsumoto:kummerred}*{Section 4.3.1 and 4.3.2}) 
are almost identical to the ones given in Theorem \ref{thm:projective equation of supersingular Kummer}. 
(We have examples where types $16 A_1, 4 D_4, 2 D_8$ occur. We do not have examples with $1 D_{16}, 2 E_8$.)

\subsection*{Acknowledgments}
The idea of considering abelian-like coverings to generalize Schr\"oer's construction
came out from discussions with Shigeyuki Kondo.
I thank him heartily.

I thank Kazuhiro Fujiwara, Gebhard Martin, Yukiyoshi Nakkajima, Stefan Schr\"oer, and Fuetaro Yobuko
for helpful comments and discussions.


\begin{bibdiv}
	\begin{biblist}
		\bib{Artin:wild2}{article}{
  author={Artin, M.},
  title={Wildly ramified $Z/2$ actions in dimension two},
  journal={Proc. Amer. Math. Soc.},
  volume={52},
  date={1975},
  pages={60--64},
  issn={0002-9939},
}

\bib{Artin:RDP}{article}{
  author={Artin, M.},
  title={Coverings of the rational double points in characteristic $p$},
  conference={ title={Complex analysis and algebraic geometry}, },
  book={ publisher={Iwanami Shoten, Tokyo}, },
  date={1977},
  pages={11--22},
}

\bib{Badescu:surfaces}{book}{
  author={B\u {a}descu, Lucian},
  title={Algebraic surfaces},
  series={Universitext},
  note={Translated from the 1981 Romanian original by Vladimir Ma\c {s}ek and revised by the author},
  publisher={Springer-Verlag, New York},
  date={2001},
  pages={xii+258},
  isbn={0-387-98668-5},
}

\bib{Bombieri--Mumford:III}{article}{
  author={Bombieri, E.},
  author={Mumford, D.},
  title={Enriques' classification of surfaces in char. $p$. III},
  journal={Invent. Math.},
  volume={35},
  issn={0020-9910},
  date={1976},
  pages={197--232},
}

\bib{Illusie:deRham--Witt}{article}{
  author={Illusie, Luc},
  title={Complexe de de\thinspace Rham-Witt et cohomologie cristalline},
  language={French},
  journal={Ann. Sci. \'{E}cole Norm. Sup. (4)},
  volume={12},
  date={1979},
  number={4},
  pages={501--661},
  issn={0012-9593},
}

\bib{Ito--Ito--Koshikawa:CMliftings}{article}{
  author={Ito, Kazuhiro},
  author={Ito, Tetsushi},
  author={Koshikawa, Teruhisa},
  title={CM liftings of $K3$ surfaces over finite fields and their applications to the Tate conjecture},
  journal={Forum Math. Sigma},
  volume={9},
  date={2021},
  pages={Paper No. e29, 70},
}

\bib{Katsura:Kummer2}{article}{
  author={Katsura, Toshiyuki},
  title={On Kummer surfaces in characteristic $2$},
  conference={ title={Proceedings of the International Symposium on Algebraic Geometry}, address={Kyoto Univ., Kyoto}, date={1977}, },
  book={ publisher={Kinokuniya Book Store, Tokyo}, },
  date={1978},
  pages={525--542},
}

\bib{Katsura--Takeda:quotients}{article}{
  author={Katsura, Toshiyuki},
  author={Takeda, Y.},
  title={Quotients of abelian and hyperelliptic surfaces by rational vector fields},
  journal={J. Algebra},
  volume={124},
  date={1989},
  number={2},
  pages={472--492},
  issn={0021-8693},
}

\bib{Kim--MadapusiPera:2-adic}{article}{
  author={Kim, Wansu},
  author={Madapusi Pera, Keerthi},
  title={2-adic integral canonical models},
  journal={Forum Math. Sigma},
  volume={4},
  date={2016},
  pages={e28, 34},
}

\bib{Kondo--Schroer:Kummer2}{article}{
  author={Kond\={o}, Shigeyuki},
  author={Schr\"{o}er, Stefan},
  title={Kummer surfaces associated with group schemes},
  journal={Manuscripta Math.},
  volume={166},
  date={2021},
  number={3-4},
  pages={323--342},
  issn={0025-2611},
}

\bib{Liedtke--Martin--Matsumoto:RDPtors}{article}{
  author={Liedtke, Christian},
  author={Martin, Gebhard},
  author={Matsumoto, Yuya},
  title={Torsors over the rational double points in characteristic $p$},
  year={2021},
  eprint={https://arxiv.org/abs/2110.03650v1},
  note={to appear in Ast\'erisque},
}

\bib{MadapusiPera:TateK3}{article}{
  author={Madapusi Pera, Keerthi},
  title={The Tate conjecture for K3 surfaces in odd characteristic},
  journal={Invent. Math.},
  volume={201},
  date={2015},
  number={2},
  pages={625--668},
  issn={0020-9910},
}

\bib{Matsumoto:rdpderi}{article}{
  author={Matsumoto, Yuya},
  title={Purely inseparable coverings of rational double points in positive characteristic},
  journal={J. Singul.},
  volume={24},
  pages={79--95},
  year={2022},
}

\bib{Matsumoto:k3alphap}{article}{
  author={Matsumoto, Yuya},
  title={$\mu _p$- and $\alpha _p$-actions on K3 surfaces in characteristic $p$},
  volume={32},
  date={2023},
  pages={271--322},
  year={2023},
  journal={J. Algebraic Geom.},
}

\bib{Matsumoto:k3rdpht}{article}{
  author={Matsumoto, Yuya},
  title={Inseparable maps on $W_n$-valued Ext groups of non-taut rational double point singularities and the height of K3 surfaces},
  year={2023},
  journal={J. Commut. Algebra},
  volume={15},
  number={3},
  pages={377--404},
}

\bib{Matsumoto:kummerred}{article}{
  author={Matsumoto, Yuya},
  title={Supersingular reduction of Kummer surfaces in residue characteristic $2$},
  year={2023},
  eprint={https://arxiv.org/abs/2302.09535v1},
}

\bib{Nikulin:Kummersurfaces}{article}{
  author={Nikulin, V. V.},
  title={Kummer surfaces},
  language={Russian},
  journal={Izv. Akad. Nauk SSSR Ser. Mat.},
  volume={39},
  date={1975},
  number={2},
  pages={278--293, 471},
  issn={0373-2436},
  note={English translation: Math. USSR. Izv. {\bf 9} (1975), no. 2, 261--275 (1976).},
}

\bib{Ogus:crystallinetorelli}{article}{
  author={Ogus, Arthur},
  title={A crystalline Torelli theorem for supersingular $K3$\ surfaces},
  conference={ title={Arithmetic and geometry, Vol. II}, },
  book={ series={Progr. Math.}, volume={36}, publisher={Birkh\"auser Boston, Boston, MA}, },
  date={1983},
  pages={361--394},
}

\bib{Rudakov--Shafarevich:finite}{article}{
  author={Rudakov, A. N.},
  author={Shafarevich, I. R.},
  title={Surfaces of type $K3$\ over fields of finite characteristic},
  language={Russian},
  conference={ title={Current problems in mathematics, Vol. 18}, },
  book={ publisher={Akad. Nauk SSSR, Vsesoyuz. Inst. Nauchn. i Tekhn. Informatsii, Moscow}, },
  date={1981},
  pages={115--207},
  note={English translation: J. of Soviet Math. {\bf 22} (1983), no. 4, 1476--1533.},
}

\bib{Schroer:Kummer2}{article}{
  author={Schr\"{o}er, Stefan},
  title={Kummer surfaces for the self-product of the cuspidal rational curve},
  journal={J. Algebraic Geom.},
  volume={16},
  date={2007},
  number={2},
  pages={305--346},
  issn={1056-3911},
}

\bib{SGA3-1}{collection}{
  title={Sch\'{e}mas en groupes (SGA 3). Tome I. Propri\'{e}t\'{e}s g\'{e}n\'{e}rales des sch\'{e}mas en groupes},
  language={French},
  series={Documents Math\'{e}matiques (Paris) [Mathematical Documents (Paris)]},
  volume={7},
  editor={Gille, Philippe},
  editor={Polo, Patrick},
  note={S\'{e}minaire de G\'{e}om\'{e}trie Alg\'{e}brique du Bois Marie 1962--64. [Algebraic Geometry Seminar of Bois Marie 1962--64]; A seminar directed by M. Demazure and A. Grothendieck with the collaboration of M. Artin, J.-E. Bertin, P. Gabriel, M. Raynaud and J-P. Serre; Revised and annotated edition of the 1970 French original},
  publisher={Soci\'{e}t\'{e} Math\'{e}matique de France, Paris},
  date={2011},
  pages={xxviii+610},
  label={SGA3-1},
}

\bib{Serre:topologie}{article}{
  author={Serre, Jean-Pierre},
  title={Sur la topologie des vari\'{e}t\'{e}s alg\'{e}briques en caract\'{e}ristique $p$},
  language={French},
  conference={ title={Symposium internacional de topolog\'{\i }a algebraica International symposium on algebraic topology}, },
  book={ publisher={Universidad Nacional Aut\'{o}noma de M\'{e}xico and UNESCO, M\'{e}xico}, },
  date={1958},
  pages={24--53},
}

\bib{Shioda:Kummer2}{article}{
  author={Shioda, Tetsuji},
  title={Kummer surfaces in characteristic $2$},
  journal={Proc. Japan Acad.},
  volume={50},
  date={1974},
  pages={718--722},
  issn={0021-4280},
}


	\end{biblist}
\end{bibdiv}
\end{document}